\documentclass[a4paper,oneside,12pt,reqno]{amsart}
\usepackage[utf8]{inputenc}
\usepackage[T1]{fontenc}
\usepackage[english]{babel}
\usepackage[margin=3cm]{geometry}
\usepackage{amsmath}
\usepackage{amsthm}
\usepackage[theoremfont]{newpx}
\usepackage{enumitem}
\usepackage[colorlinks,allcolors=blue]{hyperref}
\usepackage[capitalize,noabbrev]{cleveref}
\usepackage{orcidlink}
\usepackage{tikz-cd}
\usepackage{todonotes}
\NewDocumentEnvironment{enumList}{o}{
    \IfValueTF{#1}{
        \begin{enumerate}[label=\textrm{(\alph*)},#1]
    }{
    \begin{enumerate}[label=\textrm{(\alph*)}]
    }
}{
    \end{enumerate}
}
\NewDocumentEnvironment{enumEquiv}{o}{
    \IfValueTF{#1}{
        \begin{enumerate}[label=\textrm{(\roman*)},#1]
    }{
        \begin{enumerate}[label=\textrm{(\roman*)}]
    }
}{
    \end{enumerate}
}
\NewDocumentEnvironment{enumItem}{o}{
    \IfValueTF{#1}{
        \begin{enumerate}[label=\textrm{(\roman*)},#1]
    }{
        \begin{enumerate}[label=\textrm{(\roman*)}]
    }
}{
    \end{enumerate}
}

\numberwithin{equation}{section}

\theoremstyle{plain}
\newtheorem{theorem}{Theorem}[section]
\newtheorem{proposition}[theorem]{Proposition}
\newtheorem{lemma}[theorem]{Lemma}
\newtheorem{corollary}[theorem]{Corollary}

\theoremstyle{definition}
\newtheorem{definition}[theorem]{Definition}
\newtheorem{example}[theorem]{Example}
\newtheorem{remark}[theorem]{Remark}

\AddToHook{env/proposition/begin}{\crefalias{theorem}{proposition}}
\AddToHook{env/lemma/begin}{\crefalias{theorem}{lemma}}
\AddToHook{env/corollary/begin}{\crefalias{theorem}{corollary}}
\AddToHook{env/definition/begin}{\crefalias{theorem}{definition}}
\AddToHook{env/example/begin}{\crefalias{theorem}{example}}
\AddToHook{env/remark/begin}{\crefalias{theorem}{remark}}
\AddToHook{env/question/begin}{\crefalias{theorem}{question}}

\crefname{equation}{}{}
\crefname{enumi}{}{}

\let\originalleft\left
\let\originalright\right
\renewcommand{\left}{\mathopen{}\mathclose\bgroup\originalleft}
\renewcommand{\right}{\aftergroup\egroup\originalright}

\newcommand{\bbN}{\mathbb{N}}

\newcommand{\bbR}{\mathbb{R}}

\newcommand{\calA}{\mathcal{A}}

\newcommand{\calL}{\mathcal{L}}

\newcommand{\calS}{\mathcal{S}}

\newcommand{\calZ}{\mathcal{Z}}

\newcommand{\rmL}{\mathrm{L}}

\newcommand{\rmd}{\mathrm{d}}

\newcommand{\rmr}{\mathrm{r}}

\newcommand{\bfzero}{\mathbf{0}}
\newcommand{\bfone}{\mathbf{1}}

\newcommand{\nat}{\bbN}

\newcommand{\real}{\bbR}

\renewcommand{\to}{\longrightarrow}
\renewcommand{\mapsto}{\longmapsto}

\newcommand{\abs}[1]{|{#1}|}
\newcommand{\Abs}[1]{\left|{#1}\right|}
\newcommand{\norm}[1]{\|{#1}\|}
\newcommand{\Norm}[1]{\left\|{#1}\right\|}

\newcommand{\dset}[2]{\{#1\;|\;#2\}}
\newcommand{\Dset}[2]{\left\{#1\;\middle|\;#2\right\}}
\newcommand{\set}[1]{\{#1\}}
\newcommand{\Set}[1]{\left\{{#1}\right\}}

\newcommand{\eps}{\varepsilon} 

\newcommand{\id}{\mathrm{id}} 

\DeclareMathOperator{\ba}{ba}

\let\hat\widehat

\title{The $T$-strong duals of $L^1(T)$ and $L^\infty(T)$}
\author[F.\ Boisen]{Florian Boisen
    \orcidlink{0009-0000-7632-6320}
}
\address[F.\ Boisen]{Faculty of Mathematics, Institute of Analysis, Dresden University of Technology, 01062 Dresden, Germany}
\email{florian.boisen@tu-dresden.de}
\author[A.\ Kalauch]{Anke Kalauch}
\address[A.\ Kalauch]{Faculty of Mathematics, Institute of Analysis, Dresden University of Technology, 01062 Dresden, Germany; and Research fellow at University of the Free State, Department of Mathematics and Applied Mathematics, Faculty of Natural and Agricultural Sciences, PO Box 339, Bloemfontein 9300, South Africa}
\email{anke.kalauch@tu-dresden.de}
\author[W.\ Kuo]{Wen-Chi Kuo
    \orcidlink{0000-0002-1678-8598}
}
\address[W.\ Kuo]{School of Mathematics, University of the Witwatersrand, Private Bag 3, PO WITS 2050, South Africa; and National Institute of Theoretical and Computational Sciences (NITheCS), Johannesburg, South Africa}
\email{wen.kuo@wits.ac.za}
\author[J.\ Stennder]{Janko Stennder
    \orcidlink{0000-0003-2413-759X}
}
\address[J.\ Stennder]{Faculty of Mathematics, Institute of Analysis, Dresden University of Technology, 01062 Dresden, Germany}
\email{janko.stennder@tu-dresden.de}
\author[B.\ A. Watson]{Bruce A. Watson
    \orcidlink{0000-0003-2403-1752}
}
\address[B.\ A.\ Watson]{School of Mathematics, University of the Witwatersrand, Private Bag 3, PO WITS 2050, South Africa; National Institute of Theoretical and Computational Sciences (NITheCS), Johannesburg, South Africa; and Centre of Excellence in Mathematical and Statistical Sciences (CoE-MaSS)}
\email{bruce.watson@wits.ac.za}

\begin{document}
    \begin{abstract}
        For a conditional expectation operator $T$ on a Dedekind complete Riesz space, we give representations of the $T$-strong duals of $L^1(T)$ and $L^\infty(T)$.
        The representation for the $T$-strong dual of $L^1(T)$ follows from the known result for $L^2(T)$.
        To describe the $T$-strong dual of $L^\infty(T)$, we introduce charges on components of weak order units and develop a corresponding integration theory.
    \end{abstract}

    \subjclass[2020]{46A40, 
        46E40, 
        46G10, 
        47B60, 
        47B65  
    }
    \keywords{Riesz space; conditional expectation operator; dual space; charge; finitely-additive measure}

    \maketitle
\section{Introduction} \label{sec:intro}

Given a complete $\sigma$-finite measure space $(\Omega,\calA,\lambda)$, the space $L^\infty(\Omega,\calA,\lambda)'$ of all bounded linear functionals on the space of (equivalence classes) of essentially bounded measurable linear functionals can be identified with the space of all bounded charges, i.e., finitely additive measures, on $(\Omega,\calA)$ that are absolutely continuous with respect to $\lambda$, cf. \cite{Toland-2020,Yosida-Hewitt-1952}.
In the present paper, we provide a similar representation for the $T$-strong dual of $L^\infty(T)$ in a Riesz space setting.

A \emph{conditional Riesz triple} is a triple $(E,e,T)$ such that
\begin{enumItem}
    \item $E$ is a Dedekind complete Riesz space,
    \item $e$ is a weak order unit in $E$, i.e., the band generated by $e$ coincides with $E$,
    \item $T \colon E \to E$ is a \emph{conditional expectation operator}, i.e., an order continuous positive linear projection such that $T(e) = e$ and the range $R(T)$ of $T$ is a Dedekind complete Riesz subspace of $E$, and
    \item $T$ is strictly positive, i.e., $T[E_+ \setminus \set{0}] \subseteq E_+ \setminus \set{0}$. 
\end{enumItem}
Conditional expectations on Riesz spaces were introduced in \cite{Kuo-Labuschagne-Watson-2005} and conditional Riesz triples in \cite{Azouzi-Ben-Amor-Cherif-Masmoudi-2024}.

In \cite{Kuo-Labuschagne-Watson-2005}, the conditional expectation $T$ is extended to its maximal domain, called the \emph{$T$-universal completion}, denoted by $L^1(T)$.
So far, $L^1(T)$ is defined constructively as a subspace of the universal completion of $E$.
Alternatively, in \cref{sec:T-universal-completion-Lp-spaces}, we give an axiomatic definition and show that it is unique up to Riesz isomorphism. 

For $p = 2$ and $p = \infty$, the spaces $L^p(T)$ were introduced in \cite{Labuschagne-Watson-2010} as
\begin{align*}
    L^2(T) & \coloneqq \Dset{f \in L^1(T)}{\abs{f}^2 \in L^1(T)}, \\
    L^\infty(T) & \coloneqq \Dset{f \in L^1(T)}{\exists g \in R(T)_+ : \abs{f} \leq g}.
\end{align*}
Using functional calculus, a definition for $L^p(T)$ for $p \in (1,\infty)$ is given in \cite{Azouzi-Trabelsi-2017} as
\[
    L^p(T) \coloneqq \Dset{f \in L^1(T)}{\abs{f}^p \in L^1(T)}.
\]
Approriate $R(T)$-valued norms $\norm{\cdot}_{T,p}$ on these spaces are defined similarly to norms on classical $L^p$-spaces, see \cref{sec:T-universal-completion-Lp-spaces} below.
In \cite{Kalauch-Kuo-Watson-2024b}, the notion of a \emph{$T$-strong dual space} of $L^2(T)$ was introduced.
Analogously, we consider the $T$-strong dual of $L^p(T)$ for $p \in [1,\infty]$, defined as
\[
    \hat{L^p(T)} \coloneqq \Dset{\varphi \in \rmL^\rmr(L^p(T),R(T))}{\begin{aligned}
        & \varphi \text{ is $R(T)$-homogeneous}, \\
        & \exists k \in R(T)_+ \,\forall f \in L^p(T) : \abs{\varphi(f)} \leq k \norm{f}_{T,p}
    \end{aligned}},
\]
where $\rmL^r(L^p(T),R(T))$ denotes the space of all regular linear operators from $L^p(T)$ to $R(T)$, i.e., all linear operators that can be expressed as a difference of two positive linear operators.
Note that, as $R(T)$ is Dedekind complete, $\rmL^r(L^p(T),R(T))$ coincides with the space of all order bounded linear operators, cf. \cite{Aliprantis-Burkinshaw-2006}.

In \cite{Kalauch-Kuo-Watson-2024b}, it was established that the $T$-strong dual of $L^2(T)$ is isometrically Riesz isomorphic to $L^2(T)$.
In the present paper, we give representations for the $T$-strong duals of $L^1(T)$ and $L^\infty(T)$.

As a preparation to describe the $T$-strong dual of $L^\infty(T)$, we introduce charges on components of Riesz spaces in \cref{sec:charges} and develop an integration theory with respect to these charges in \cref{sec:integration}. 
In \cref{sec:duals}, we use the representation of $\hat{L^2(T)}$ to conclude that $\hat{L^1(T)}$ is isometrically Riesz isomorphic to $L^\infty(T)$, see \cref{thm:dual-of-l1} below.
Moreover, using the integration theory established in \cref{sec:integration}, we show that $\hat{L^\infty(T)}$ is isometrically Riesz isomorphic to the space of $T$-absolutely continuous order bounded signed charges on the components of the weak order unit $e$, see \cref{thm:dual-of-linfty} below.
In \cref{sec:application}, we consider an illustrative example and link a classical method to obtain the dual of $L^\infty(T)$ with our new approach.

%
%
%
%

  \section{Preliminaries}

For basic terminology on Riesz spaces, we refer to \cite{Aliprantis-Burkinshaw-2006,Luxemburg-Zaanen-1971,Zaanen-1997}.
We list notation and results that will be used throughout this paper.
For a set $M$, we denote by $\id \coloneqq \id_M \colon M \to M$ the identity map on $M$.

%

Let $E$ be a Riesz space.
In order to distinguish between suprema taken in different Riesz spaces, we occasionally denote, for $A \subseteq E$, by $\sup_E A$ the supremum of the set $A$ in $E$.
As usual, we denote $\sup A = \sup_E A$ if no confusion arises.
For a projection band $B$ in $E$, the corresponding band projection is denoted by $P_B \colon E \to E$.
For $f \in E_+$, we denote by $A_f$, resp. $B_f$ the principal ideal, resp. principal band generated by $f$ in $E$.
For a projection band $B_f$, we abbreviate $P_f \coloneqq P_{B_f}$.
Recall that, if $E$ is Dedekind complete, then every band is a projection band, cf. \cite[Theorem 1.42]{Aliprantis-Burkinshaw-2006}.
We will use the following standard result on the lattice-structure of the space of all regular linear operators, cf. \cite[Theorem 1.18]{Aliprantis-Burkinshaw-2006}.

\begin{theorem}[Riesz--Kantorovich] \label{thm:riesz-kantorovich}
    Let $E$ and $F$ be Riesz spaces with $F$ Dedekind complete.
    Then $\rmL^\rmr(E,F)$ is a Dedekind complete Riesz space and, for all $T \in \rmL^\rmr(E,F)$ and $f \in E_+$, we have
    \begin{align*}
        T^+(f) & = \sup \Dset{T(g)}{g \in E, 0 \leq g \leq f}, \\
        \Abs{T}(f) & = \sup \Dset{\abs{T(g)}}{g \in E, -f \leq g \leq f}.
    \end{align*}
\end{theorem}

We need the subsequent version of the Freudenthal spectral theorem, cf. \cite[Theorems 33.2 and 33.3]{Zaanen-1997}.
For a weak order unit $e \in E$, an element $f \in E$ is called an \emph{$e$-step function} if there exist $\alpha_1,\dots,\alpha_n \in \real$ and pairwise disjoint components $p_1,\dots,p_n \in E_+$ of $e$ such that $f = \sum_{i=1}^n \alpha_i p_i$.

\begin{theorem} \label{thm:freudenthal}
    Let $E$ be a Dedekind complete Riesz space and $e \in E_+ \setminus \set{0}$.
    \begin{enumList}
        \item \label{it:thm:freudenthal:uniform}
        For every $f \in (A_e)_+$, there exists an increasing sequence $(s_n)_{n \in \nat}$ of positive $e$-step functions that converges $e$-uniformly to $f$.
        
        \item \label{it:thm:freudenthal:order}
        For every $f \in (B_e)_+$, there exists an increasing sequence $(s_n)_{n \in \nat}$ of positive $e$-step functions that converges in order to $f$.
    \end{enumList} 
\end{theorem}


A Dedekind complete Riesz space $E$ is called \emph{universally complete} if every non-empty set of mutually disjoint elements in $E$ has a supremum.
Recall that a Riesz subspace $D$ of a Riesz space $E$ is called \emph{order dense} if, for all $f \in E_+ \setminus \set{0}$, there exists $g \in D_+ \setminus \set{0}$ such that $g \leq f$.
Every Archimedean Riesz space $E$ can be embedded as an order dense Riesz subspace into a (unique up to Riesz isomorphism) universally complete Riesz space $E^u$, called the \emph{universal completion} of $E$.
The space $E^u$ can be endowed with the structure of an $f$-algebra.
Moreover, if $e \in E$ is a weak order unit of $E$, then $e$ is a weak order unit of $E^u$ and the $f$-algebra structure on $E^u$ can be constructed such that $e$ is an algebraic unit of $E^u$.
For details on the construction of the universal completion and its $f$-algebra structure, we refer to \cite[Section 50]{Luxemburg-Zaanen-1971} or \cite[Chapter 7]{Aliprantis-Burkinshaw-2003}.
Two elements $f,g \in E^u$ are disjoint if and only if $fg = 0$, cf. \cite[page 228]{Huijsmans-dePagter-1982}.

We need the following extension result on Riesz isomorphisms, cf. \cite[Theorem 7.19]{Aliprantis-Burkinshaw-2003}.

\begin{lemma} \label{thm:extend-riesz-iso}
    Let $E_1,E_2$ be universally complete Riesz spaces, $D_1 \subseteq E_1$ and $D_2 \subseteq E_2$ order dense Riesz subspaces, and $\Phi \colon D_1 \to D_2$ a Riesz isomorphism.
    Then there exists a unique Riesz isomorphism from $E_1$ to $E_2$ that extends $\Phi$.
\end{lemma}

For an element $f \in E^u$, an element $g \in E^u$ is called a \emph{partial inverse of $f$} if $fg = gf = P_{\abs{f}}(e)$.
We call $g$ a \emph{canonical partial inverse of $f$} if $g \in B_{\abs{f}}$.
It follows from the combination of \cite[Theorem 5]{Huijsmans-dePagter-1986} and \cite[Remark 3.3]{Kuo-Rogans-Watson-2017} that every element $f \in E^u$ has a unique canonical partial inverse, which is positive provided that $f$ is positive, see also \cite[Theorem 2.6]{Kalauch-Kuo-Watson-2024}.
We denote by $f^{-1}$ the canonical partial inverse of $f$.

In \cite{Buskes-daPagter-vanRooij-1991}, Buskes, de Pagter, and van Rooij introduced a functional calculus on Riesz spaces.
A different type of a functional calculus was given by Grobler in \cite{Grobler-2014}.
In \cite{Azouzi-Trabelsi-2017}, it is shown that both functional calculi coincide for continuous functions.
Moreover, if $E$ is Dedekind complete, it follows from \cite[Theorem 3.7]{Buskes-daPagter-vanRooij-1991} that $f(x)$ exists for every $x \in E$ and every positively homogeneuous function $f \colon \real \to \real$.
In particular, for $x \in E_+$ and $p \in [1,\infty)$, the $p$-th and $\frac{1}{p}$-th powers $x^p$ and $x^{\frac{1}{p}}$ can be defined using the functional calculus.
For the following lemma, we refer to \cite[page 805]{Azouzi-Trabelsi-2017}.

\begin{lemma} \label{thm:order-continuity-functional-rep}
    Let $E$ be a Dedekind complete Riesz space with weak order units, let $f \colon \real_+ \to \real_+$ be concave and positively homogeneous, and let $(x_\alpha)_{\alpha \in A}$ be a net in $E_+$ and $x \in E_+$ such that $x_\alpha \downarrow x$.
    Then $f(x_\alpha) \downarrow f(x)$.
\end{lemma}

Let $F$ be an $f$-algebra and $E$ an $F$-module.
An \emph{$F$-valued norm} on $E$ is a map $\norm{\cdot} \colon E \to F$ that satisfies the usual norm axioms:
\begin{enumItem}
    \item For every $x,y \in E$, one has $\norm{x + y} \leq \norm{x} + \norm{y}$.
    \item For every $x \in F$ and $y \in E$, one has $\norm{xy} = \abs{x} \norm{y}$.
    \item For every $x \in E$, one has $\norm{x} = 0$ if and only if $x = 0$.
\end{enumItem}

In \cite[Section 6]{Kuo-Labuschagne-Watson-2005}, expectation operators, i.e., conditional expectation operators with one-dimensional range, are considered.
It is shown that it suffices to consider strictly positive expectation operators as, otherwise, one can factor out the null ideal.
Next, we remark that the same procedure works for conditional expectation operators.

\begin{remark}
    Let $E$ be a Dedekind complete Riesz space with weak order unit $e \in E_+$ and let $T \colon E \to E$ be a conditional expectation with $T(e) = e$.
    Denote by
    \[
        N_T \coloneqq \Dset{f \in E}{T\abs{f} = 0}
    \]
    the \emph{null ideal} of $T$ and by $C_T \coloneqq N_T^\rmd$ the \emph{carrier} of $T$.
    Note that $N_T \subseteq \ker T$.
    Since $T$ is order continuous, $N_T$ is a band in $E$, cf. \cite[Theorem 1.61(1)]{Aliprantis-Burkinshaw-2006}.
    As $E$ is Dedekind complete, $N_T$ is a projection band, and $N_T$ and $C_T$ are Dedekind complete.
    \begin{enumList}
        \item We show that
        \[
            \tilde{T} \colon C_T \to C_T, \quad f \mapsto P_{C_T}(Tf),
        \]
        is a strictly positive conditional expectation on $C_T$.
        
        Since $T$ and $P_{C_T}$ are positive and order continuous, $\tilde{T}$ is also positive and order continuous.
        For every $f \in C_T$, we have
        \begin{align*}
            \abs{f} \wedge (n P_{C_T}(e))
            & = P_{C_T}(\abs{f}) \wedge (n P_{C_T}(e)) \\
            & = P_{C_T}(\abs{f} \wedge n e)
              \uparrow  P_{C_T}(\abs{f})
              = \abs{f},
        \end{align*}
        which shows that $P_{C_T}(e)$ is a weak order unit of $C_T$.
        Moreover, as $P_{N_T}(e) \in \ker(T)$, we obtain 
        \begin{align*}
            \tilde{T}(P_{C_T}(e))
            & = P_{C_T}(T(P_{C_T}(e)))
              = P_{C_T}(T(P_{C_T}(e) + P_{N_T}(e))) \\
            & = P_{C_T}(Te)
              = P_{C_T}(e).
        \end{align*}
        For every $f \in C_T$, we have $P_{N_T} T f \in \ker T$, thus
        \begin{align*}
            \tilde{T}^2 f
            & = P_{C_T} T P_{C_T} T f \\
            & = P_{C_T} T (P_{C_T} T f + P_{N_T} T f) \\
            & = P_{C_T} T^2 f
              = P_{C_T} T f
              = \tilde{T} f.
        \end{align*}
        Hence, $\tilde{T}$ is a projection.

        Next, we show that $R(\tilde{T})$ is a Riesz subspace of $C_T$.
        To this end, let $f \in C_T$.
        Since $R(T)$ is a Riesz subspace of $E$, there exists $g \in E$ such that $\abs{Tf} = Tg$.
        Then
        \begin{align*}
            \abs{\tilde{T} f}
            & = \abs{P_{C_T} T f}
              = P_{C_T} \abs{T f}
              = P_{C_T} T g \\
            & = P_{C_T} T (P_{C_T} g + P_{N_T} g)
              = P_{C_T} T P_{C_T} g
              = \tilde{T}(P_{C_T} g)
              \in R(\tilde{T}).
        \end{align*}
        Hence, $R(\tilde{T})$ is a Riesz subspace of $C_T$.

        Next, we show that $R(\tilde{T})$ is Dedekind complete.
        To this end, let $A \subseteq C_T$ be a non-empty subset such that $\tilde{T}[A]$ is bounded from above in $R(\tilde{T})$.
        Then there exists $g \in C_T$ such that, for all $f \in A$, we have $\tilde{T}(f) \leq \tilde{T}(g)$.
        For every $f \in A$, it follows that
        \[
            T f
            = P_{C_T}(Tf) + P_{N_T}(Tf)
            = \tilde{T} f + P_{N_T}(Tf)
            \leq \tilde{T} g + P_{N_T}(Tf).
        \]
        Applying $T$ again while using that $T$ is a projection and $N_T \subseteq \ker(T)$, we obtain $Tf \leq T(\tilde{T} g)$ for all $f \in A$.
        Hence, $T[A]$ is bounded from above in $R(T)$.
        Since $R(T)$ is Dedekind complete, there exists $h \in E$ such that $Th = \sup\nolimits_{R(T)} T[A]$.
        As $C_T$ is Dedekind complete, $\sup\nolimits_{C_T} \tilde{T}[A]$ exists and it follows that
        \begin{align*}
            \sup\nolimits_{C_T} \tilde{T}[A]
            & = \sup\nolimits_{C_T} P_{C_T}[T[A]]
              = P_{C_T}(\sup\nolimits_{R(T)} T[A])
              = P_{C_T}(T h) \\
            & = P_{C_T} T (P_{C_T} h + P_{N_T} h)
              = P_{C_T} T P_{C_T} h
              = \tilde{T}(P_{C_T} h)
              \in R(\tilde{T}).
        \end{align*}
        Consequently, $\sup\nolimits_{R(\tilde{T})} \tilde{T}[A] = \sup\nolimits_{C_T} \tilde{T}[A]$ and, thus, $R(\tilde{T})$ is Dedekind complete.

        Finally, we show that $\tilde{T}$ is strictly positive.
        Let $f \in (C_T)_+$ be such that $\tilde{T}f = 0$.
        Then $P_{C_T}(Tf) = 0$, which implies $Tf \in N_T \subseteq \ker T$.
        As $T$ is a projection, we obtain $Tf = T(Tf) = 0$, thus $f \in N_T$.
        Then $f \in C_T \cap N_T$, which gives $f = 0$.
        Hence, $\tilde{T}$ is strictly positive.

        \item Note that $E/N_T$ is also a Riesz space, where, for $f,g \in E$, one defines $f + N_T \leq g + N_T$ if there exist $h_f,h_g \in N_T$ such that $f + h_f \leq g + h_g$, cf. \cite[Theorem 2.22]{Aliprantis-Burkinshaw-2006}.
        Moreover, the map
        \[
            \Phi \colon C_T \to E/N_T, \quad f \mapsto f + N_T,
        \]
        is a Riesz isomorphism with inverse
        \[
            \Phi^{-1} \colon E/N_T \to C_T, \quad f \mapsto P_{C_T}(f).
        \]
        Hence, it follows with the aid of (a) that
        \[
            \Phi \circ \tilde{T} \circ \Phi^{-1} \colon E/N_T \to E/N_T, \quad f + N_T \mapsto Tf + N_T,
        \]
        is a strictly positive conditional expectation on $E/N_T$.
    \end{enumList}
\end{remark}

Occasionally, we consider band projections onto principal bands generated by elements in $R(T)$ and apply them to elements in $R(T)$.
In this case, one might ask whether the principal band is generated as a band in $R(T)$ or as a band in $E$.
We show that both approaches give the same result in a general setting and then apply this to conditional Riesz triples.

\begin{proposition} \label{thm:projections-E-vs-D}
    Let $E$ be a Riesz space with the principal projection property, $D \subseteq E$ a sequentially order closed Riesz subspace, and $f \in D$.
    Denote by $B_{f,E} \subseteq E$ and $B_{f,D} \subseteq D$ the principal bands generated by $f$ in $E$ and $D$, respectively, and let $P_{f,E} \colon E \to B_{f,E}$ and $P_{f,D} \colon D \to B_{f,D}$ be their respective band projections.
    Then:
    \begin{enumList}
        \item \label{it:thm:projections-E-vs-D:projections}
        $P_{f,E}|_{D} = P_{f,D}$.
        
        \item \label{it:thm:projections-E-vs-D:bands}
        $B_{f,E} \cap D = B_{f,D}$.
    \end{enumList}
\end{proposition}

\begin{proof}
    \ref{it:thm:projections-E-vs-D:projections}
    Let $g \in D_+$.
    Note that $P_{f,E}(g) \leq P_{f,D}(g)$.
    Since $D$ is a sequentially order closed Riesz subspace of $E$, we have
    \[
        P_{f,E}(g)
        = \sup\Dset{g \wedge n \abs{f}}{n \in \nat}
        \in D.
    \]
    For every $m \in \nat$, we have $g \wedge m \abs{f} \leq P_{f,E}(g) \in D$, which implies $P_{f,D}(g) \leq P_{f,E}(g)$.
    Hence, $P_{f,E}(g) = P_{f,D}(g)$.

    \smallskip

    \ref{it:thm:projections-E-vs-D:bands}
    It follows from \ref{it:thm:projections-E-vs-R(T):projections} that
    \begin{align*}
        B_{f,E} \cap D
        & = \Dset{g \in D}{P_{f,E}(\abs{g}) = \abs{g}} \\
        & = \Dset{g \in D}{P_{f,D}(\abs{g}) = \abs{g}}
          = B_{f,D}.
        \qedhere
    \end{align*}
\end{proof}

\begin{corollary} \label{thm:projections-E-vs-R(T)}
    Let $(E,e,T)$ be a conditional Riesz triple and let $f \in R(T)$.
    Denote by $B_{f,E} \subseteq E$ and $B_{f,R(T)} \subseteq R(T)$ the principal bands generated by $f$ in $E$ and $R(T)$, respectively, and let $P_{f,E} \colon E \to B_{f,E}$ and $P_{f,R(T)} \colon R(T) \to B_{f,R(T)}$ be their respective band projections.
    Then:
    \begin{enumList}
        \item \label{it:thm:projections-E-vs-R(T):projections}
        $P_{f,E}|_{R(T)} = P_{f,R(T)}$.
        
        \item \label{it:thm:projections-E-vs-R(T):bands}
        $B_{f,E} \cap R(T) = B_{f,R(T)}$.
    \end{enumList}
\end{corollary}

\begin{proof}
    Since $T$ is an order continuous projection, $R(T)$ is order closed.
    Hence, the claim follows directly from \cref{thm:projections-E-vs-D}.
\end{proof}

In view of \cref{thm:projections-E-vs-R(T)}, we assume that band projections $P_f$ for $f \in R(T)$ are defined on $E$.

The following lemma is taken from \cite[Corollary 2.3]{Kuo-Labuschagne-Watson-2005-b}.
Since the reference is not widely accessible, we provide a complete proof.

\begin{lemma} \label{thm:proj-ineq}
    Let $(E,e,T)$ be a conditional Riesz triple and let $f \in E_+$.
    Then $P_f \leq P_{Tf}$.
\end{lemma}

\begin{proof}
    It suffices to show that $f = P_{Tf}f$, as this implies that $f \in B_{Tf}$ and, thus, $P_f \leq P_{Tf}$.
    By \cref{thm:projections-E-vs-R(T)}, we have $P_{Tf}e \in R(T)$.
    Denote $g \coloneqq e - P_{Tf}e \in R(T)_+$.
    As $Tf \in B_{Tf}$ and $g \in B_{Tf}^d$, we have $Tf \wedge g = 0$, thus 
    \[
        0
        \leq T(f \wedge g)
        \leq Tf \wedge Tg
        = Tf \wedge g
        = 0.
    \]
    This implies $T(f \wedge g) = 0$.
    Since $T$ is strictly positive, we obtain $f \wedge g = 0$, thus $f \cdot g = 0$, where the product is taken in $E^u$.
    Then we have $f - P_{Tf} f = f \cdot g = 0$, thus $f = P_{Tf} f$, as required.
\end{proof}

\section{The \texorpdfstring{$T$}{T}-universal completion and properties of \texorpdfstring{$L^p$}{Lp}-type spaces} \label{sec:T-universal-completion-Lp-spaces}

We give a brief exposition on the $T$-universal completion and $L^p$-type spaces in Riesz spaces.
In \cite{Kuo-Labuschagne-Watson-2005}, the $T$-universal completion is defined constructively.
We give an axiomatic definition of the $T$-universal completion and show that it is unique up to Riesz isomorphism. 
In this section, we consider the universal completion $E^u$ of an Archimedean Riesz space $E$ together with its embedding $j \colon E \to E^u$, i.e., $j$ is a bipositive linear map such that $j[E]$ is an order dense Riesz subspace of $E^u$.

\begin{definition}
    Let $(E,e,T)$ be a conditional Riesz triple and $E^u$ the universal completion of $E$ with its embedding $j \colon E \to E^u$.
    \begin{enumList}
        \item $E$ is called \emph{$T$-universally complete} if every non-empty upwards directed subset $A \subseteq E$ such that $j[T[A]]$ is bounded from above in $E^u$ has a supremum in $E$.
        The conditional Riesz triple $(E,e,T)$ is called \emph{$T$-universally complete} if $E$ is $T$-universally complete.

        \item A \emph{$T$-universal completion} of $(E,e,T)$ is a conditional Riesz triple $(\tilde{E},\tilde{e},\tilde{T})$ together with a bipositive linear map $i \colon E \to \tilde{E}$ with the following properties:
        \begin{enumItem}
            \item $\tilde{E}$ is $\tilde{T}$-universally complete.
            \item $i[E]$ is an order dense Riesz subspace of $\tilde{E}$.
            \item $i(e) = \tilde{e}$.
            \item $\tilde{T} \circ i = i \circ T$.
        \end{enumItem}
    \end{enumList}
\end{definition}

Note that, as $i[E]$ is an order dense Riesz subspace of $\tilde{E}$, the map $i$ is an order continuous Riesz homomorphism, cf. \cite[Theorem 23.2]{Zaanen-1997}.
As the universal completion is unique up to Riesz isomorphism, $T$-universal completeness does not depend on the concrete choice of the universal completion.
Moreover, if $(\tilde{E},\tilde{e},\tilde{T})$ is a $T$-universal completion of $(E,e,T)$ with embedding $i \colon E \to \tilde{E}$, then $i[E]$ is an ideal in $\tilde{E}$, cf. \cite[Theorem 2.31]{Aliprantis-Burkinshaw-2006}.

\begin{theorem} \label{thm:T-universal-completion}
    Let $(E,e,T)$ be a conditional Riesz triple.
    \begin{enumList}
        \item \label{it:thm:T-universal-completion:existence}
        There exists a $T$-universal completion of $(E,e,T)$.
        
        \item \label{it:thm:T-universal-completion:unique}
        Let $(\tilde{E}_1,\tilde{e}_1,\tilde{T}_1)$ and $(\tilde{E}_2,\tilde{e}_2,\tilde{T}_2)$ be $T$-universal completions of $(E,e,T)$ with respective embeddings $i_1 \colon E \to \tilde{E}_1$ and $i_2 \colon E \to \tilde{E}_2$.
        Then there exists a Riesz isomorphism $\Phi \colon \tilde{E}_1 \to \tilde{E}_2$ such that $i_2 = \Phi \circ i_1$.
        In particular, the following diagram commutes:
        \begin{center}
            \begin{tikzcd}
                \tilde{E}_1 \ar[bend left]{rr}{\Phi} \ar[hookleftarrow]{r}{i_1} \ar{d}{\tilde{T}_1} & E \ar[hookrightarrow]{r}{i_2} \ar{d}{T} & \tilde{E}_2 \ar{d}{\tilde{T}_2} \\
                \tilde{E}_1 \ar[bend right]{rr}[swap]{\Phi} \ar[hookleftarrow]{r}{i_1} & E \ar[hookrightarrow]{r}{i_2} & \tilde{E}_2
            \end{tikzcd}
        \end{center}
    \end{enumList}
\end{theorem}

\begin{proof}
    \ref{it:thm:T-universal-completion:existence}
    This is shown in \cite[Section 5]{Kuo-Labuschagne-Watson-2005}.

    \smallskip

    \ref{it:thm:T-universal-completion:unique}
    Let $\tilde{E}_1^u$ and $\tilde{E}_2^u$ be the respective universal completions of $\tilde{E}_1$ and $\tilde{E}_2$ with their respective embeddings $j_1 \colon \tilde{E}_1 \to \tilde{E}_1^u$ and $j_2 \colon \tilde{E}_2 \to \tilde{E}_2^u$.
    Since $i_1[E]$ is an order dense Riesz subspace of $\tilde{E}_1$ and $j_1[\tilde{E}_1]$ is an order dense Riesz subspace of $\tilde{E}_1^u$, it follows that $j_1[i_1[E]]$ is an order dense Riesz subspace of $\tilde{E}_1^u$.
    Likewise, $j_2[i_2[E]]$ is an order dense Riesz subspace of $\tilde{E}_2^u$.
    Since $j_2 \circ i_2 \circ i_1^{-1} \circ j_1^{-1} \colon j_1[i_1[E]] \to j_2[i_2[E]]$ is a Riesz isomorphism, there exists by \cref{thm:extend-riesz-iso} a Riesz isomorphism $\Psi \colon \tilde{E}_1^u \to \tilde{E}_2^u$ that extends $j_2 \circ i_2 \circ i_1^{-1} \circ j_1^{-1}$.
    In particular, we then have $\Psi \circ j_1 \circ i_1 = j_2 \circ i_2$.

    We claim that $\Psi[j_1[\tilde{E}_1]] = j_2[\tilde{E}_2]$. 
    This will then imply that $\Phi \coloneqq j_2^{-1} \circ \Psi \circ j_1 \colon \tilde{E}_1 \to \tilde{E}_2$ is a Riesz isomorphism satisfying
    \[
        \Phi \circ i_1
        = j_2^{-1} \circ \Psi \circ j_1 \circ i_1
        = j_2^{-1} \circ j_2 \circ i_2
        = i_2,
    \]
    which will complete the proof.

    ($\subseteq$)
    Let $x \in \tilde{E}_1$.
    Since $i_1[E]$ is order dense in $\tilde{E}_1$, $i_1$ is an order continuous Riesz homomorphism, thus 
    \begin{align*}
        \Psi(j_1(x))
        & = \Psi\left( j_1 \left( \sup\Dset{i_1(z)}{z \in E_+, i_1(z) \leq x} \right) \right) \\
        & = \sup\Dset{\Psi(j_1(i_1(z)))}{z \in E_+, i_1(z) \leq x} \\
        & = \sup\Dset{j_2(i_2(z))}{z \in E_+, i_1(z) \leq x}.
    \end{align*}
    For every $z \in E_+$ with $i_1(z) \leq x$, we have
    \[
        j_2(\tilde{T}_2(i_2(z)))
        = j_2(i_2(T(z)))
        = \Psi(j_1(i_1(T(z))))
        = \Psi(j_1(\tilde{T}_1(i_1(z))))
        \leq \Psi(j_1(\tilde{T}_1(x))).
    \]
    Hence, the set $j_2[\tilde{T}_2[\dset{i_2(z)}{z \in E_+,i_1(z) \leq x}]]$ is bounded from above in $\tilde{E}_2^u$.
    Since $\tilde{E}_2$ is $\tilde{T}_2$-universally complete, we conclude that
    \[
        \sup\Dset{i_2(z)}{z \in E_+, i_1(z) \leq x}
    \]
    exists in $\tilde{E}_2$.
    Hence,
    \begin{align*}
        \Psi(j_1(x))
        & = \sup\Dset{j_2(i_2(z))}{z \in E_+, i_1(z) \leq x} \\
        & = j_2\left( \sup\Dset{i_2(z)}{z \in E_+,i_1(z) \leq x} \right)
        \in j_2[\tilde{E}_2].
    \end{align*}

    ($\supseteq$)
    This inclusion can be shown similarly.
\end{proof}

In the sequel, we denote the unique (up to Riesz isomorphism) $T$-universal completion of $(E,e,T)$ by $L^1(T)$.
If $(E,e,T)$ is a $T$-universally complete conditional Riesz triple, then $R(T)$ is universally complete and, thus, has the structure of an $f$-algebra.
Hereby, $e$ is both an algebraic and a weak order unit.
Occasionally, we write $f \cdot g$ instead of $fg$ for $f,g \in R(T)$ if it provides clarity.
$E = L^1(T)$ is an $R(T)$-module, cf. \cite[Theorem 2.2]{Kuo-Rogans-Watson-2017}.
Moreover, $T$ is an \emph{averaging operator}, i.e., $T(gf) = gT(f)$ for all $g \in R(T)$ and $f \in E$, cf. \cite[Theorem 4.3]{Kuo-Labuschagne-Watson-2005}.

Using the $T$-universal completion, the spaces $L^p(T)$ can be defined as described in \cref{sec:intro}.
In the remainder of this section, we discuss properties of these spaces.
$L^\infty(T)$ is an $R(T)$-module, cf. \cite[Theorem 2.2]{Kuo-Rogans-Watson-2017}, and both a Riesz subspace and a subalgebra of $E^u$.
In particular, $L^\infty(T)$ is an $f$-algebra, where $e$ is both an algebraic and a weak order unit.
Moreover, $L^p(T)$ is an order ideal of $L^1(T)$, cf. \cite[Theorem 3.2]{Azouzi-Trabelsi-2017} and an $L^\infty(T)$-module, cf. \cite[Proposition 3.12]{Azouzi-Trabelsi-2017}.
For $p \in [1,\infty)$, an $R(T)$-valued norm is defined on $L^p(T)$ via
\[
    \norm{f}_{T,p} \coloneqq T\left( \abs{f}^p \right)^{\tfrac{1}{p}} \qquad \text{for } f \in L^p(T).
\]
For a proof for the triangle inequality, we refer to \cite[Theorem 3.4]{Azouzi-Trabelsi-2017}.
The remaining norm axioms are straightforward.
For the case $p = \infty$, an $R(T)$-valued norm is defined on $L^\infty(T)$ via
\[
    \norm{f}_{T,\infty} \coloneqq \inf\Dset{\alpha \in R(T)_+}{\abs{f} \leq \alpha}
    \qquad \text{for } f \in L^\infty(T),
\]
cf. \cite[Theorem 3.2]{Kuo-Rogans-Watson-2017}.
Note that, for all $p \in [1,\infty]$, the $R(T)$-valued norm $\norm{\cdot}_{T,p}$ is monotone in the sense that the inequality $\abs{f} \leq \abs{g}$ implies $\norm{f}_{T,p} \leq \norm{g}_{T,p}$ for all $f,g \in L^p(T)$.
Finally, $R(T)$ can also be endowed with an $R(T)$-valued norm via
\[
    \norm{f}_{R(T)} \coloneqq \abs{f} \qquad \text{for } f \in R(T).
\]

The following proposition is known for the case $p = 1$ and $p = 2$, cf. \cite[Lemma 3.3(iii)]{Kalauch-Kuo-Watson-2024b}.
We will only be concerned with the case $p \in \nat$.

\begin{proposition} \label{thm:p-norm-o-continuous}
    Let $(E,e,T)$ be a $T$-universally complete conditional Riesz triple and let $p \in \nat$.
    Then $\norm{\cdot}_{T,p}$ is order continuous.
\end{proposition}

\begin{proof}
    Let $(x_\alpha)_{\alpha \in A}$ be a net in $E_+$ with $x_\alpha \downarrow 0$.
    Since $p \in \nat$, we immediately obtain $x_\alpha^p \downarrow 0$ and by the order continuity of $T$, we have $T(x_\alpha^p) \downarrow 0$.
    Finally, it follows with the aid of \cref{thm:order-continuity-functional-rep} that $\norm{x_\alpha}_{T,p} = (T(x_\alpha^p))^{\frac{1}{p}} \downarrow 0$.
    Hence, $\norm{\cdot}_{T,p}$ is order continuous.
\end{proof}

We note that, as in the classical case, the norm $\norm{\cdot}_{T,\infty}$ is not order continuous, in general.
We will need the following Hölder-type inequality, cf. \cite[Theorem 3.7]{Azouzi-Trabelsi-2017} and \cite[Theorem 3.3]{Kuo-Rogans-Watson-2017}.

\begin{proposition} \label{thm:hoelder}
    Let $(E,e,T)$ be a $T$-universally complete conditional Riesz triple, $p,q \in [1,\infty]$ conjugate exponents, and $f \in L^p(T)$ and $g \in L^q(T)$.
    Then $fg \in L^1(T)$ and
    \[
        \norm{fg}_{T,1} \leq \norm{f}_{T,p} \norm{g}_{T,q}.
    \]
\end{proposition}

The following observation on the norm of components is needed later on.

\begin{proposition} \label{thm:inf-norm-of-component}
    Let $(E,e,T)$ be a $T$-universally complete conditional Riesz triple and let $p \in E_+$ be a component of $e$.
    Then $\norm{p}_{T,\infty} = P_{Tp}(e)$.
\end{proposition}

\begin{proof}
    First, we show that $\norm{p}_{T,\infty}$ is a component of $e$.
    As $e \in R(T)_+$ and $0 \leq p \leq e$, we have $0 \leq \norm{p}_{T,\infty} \leq e$, thus $\norm{p}_{T,\infty}^2 \leq \norm{p}_{T,\infty}$.
    Conversely, since $0 \leq p \leq \norm{p}_{T,\infty}$, we obtain $p = p^2 \leq \norm{p}_{T,\infty}^2$, which implies $\norm{p}_{T,\infty} \leq \norm{p}_{T,\infty}^2$.
    Hence, $\norm{p}_{T,\infty} = \norm{p}_{T,\infty}^2$.
    This gives $\norm{p}_{T,\infty} (e - \norm{p}_{T,\infty}) = 0$, which is equivalent to $\norm{p}_{T,\infty} \wedge (e - \norm{p}_{T,\infty}) = 0$.
    Hence, $\norm{p}_{T,\infty}$ is a component of $e$.

    Next, we show that $\norm{p}_{T,\infty} = P_{Tp}(e)$.
    Since $0 \leq Tp \leq T\norm{p}_{T,\infty} = \norm{p}_{T,\infty}$, we have $P_{Tp}(e) \leq P_{\norm{p}_{T,\infty}}(e)$.
    As $\norm{p}_{T,\infty}$ is a component of $e$, we get $P_{\norm{p}_{T,\infty}}(e) = \norm{p}_{T,\infty}$, thus $P_{Tp}(e) \leq \norm{p}_{T,\infty}$.
    By \cref{thm:proj-ineq}, we have $p = P_p(e) \leq P_{Tp}(e) \in R(T)_+$, hence $\norm{p}_{T,\infty} \leq P_{Tp}(e)$.
    We conclude that $\norm{p}_{T,\infty} = P_{Tp}(e)$.
\end{proof}

For the subsequent technical lemma, recall that $L^\infty(T)$ is an $f$-algebra and note that band projections $P_\alpha \colon L^\infty(T) \to L^\infty(T)$ for $\alpha \in L^\infty(T)$ are algebra homomorphisms, cf. \cite[Corollary 5.5]{Huijsmans-dePagter-1984}. 

\begin{lemma} \label{thm:prod-w/-band-proj-invariant}
    Let $(E,e,T)$ be a $T$-universally complete conditional Riesz triple and let $f,\alpha \in L^\infty(T)_+$ be such that there exists $g \in L^\infty(T)_+$ with $f \leq \alpha g$.
    Then $P_{\alpha}(e) f = f$.
\end{lemma}

\begin{proof}
    Since $f \in B_{\alpha g} \subseteq B_\alpha$, we have $P_{\alpha}(e) f = P_{\alpha}(e) P_{\alpha}(f) = P_{\alpha}(f) = f$.
\end{proof}

\section{\texorpdfstring{$R(T)$}{R(T)}-valued charges on components} \label{sec:charges}

We introduce $R(T)$-valued charges on components of the weak order unit.
In \cref{sec:integration}, we develop an integration theory with respect to such charges, which only makes sense if a notion of absolute continuity is assumed.
First, we investigate basic properties of charges.
$T$-absolutely continuous charges will be introduced in \cref{def:T-absolutely-continuous} below.

Given a Riesz space $E$ and a weak order unit $e$, we denote by $C_e$ the set of all components of $e$.
The following standard example motivates the subsequent definition of charges on $C_e$.

\begin{example} \label{ex:charge-motivation}
    Let $(\Omega,\calA)$ be a measurable space and consider the Riesz space $E = \calL^0(\Omega,\calA)$ of all $\calA$-measurable functions with the order defined pointwise everywhere.
    The constant function $\bfone \colon \omega \mapsto 1$ is a weak order unit of $E$ and the components of $\bfone$ are precisely the characteristic functions of measurable sets, i.e.,
    \[
        \bfone_A \colon \Omega \to \real, \quad \omega \mapsto
        \begin{cases}
            1 & \text{if } \omega \in A, \\
            0 & \text{if } \omega \not\in A,
        \end{cases}
    \]
    for $A \in \calA$.
    In particular, the map
    \[
        \calA \to C_\bfone, \quad A \mapsto \bfone_A,
    \]
    is a bijection.
    Note that two sets $A,B \in \calA$ are disjoint if and only if $\bfone_A \wedge \bfone_B = 0$, i.e., if and only if the components $\bfone_A$ and $\bfone_B$ are disjoint in the Riesz space $E$.
\end{example}

\begin{definition}
    Let $E$ and $F$ be Riesz spaces and $e \in E_+$ a weak order unit of $E$.
    A \emph{signed $F$-valued charge} on $C_e$ is a map $\mu \colon C_e \to F$ with the following properties:
    \begin{enumItem}
        \item $\mu(0) = 0$.
        \item For all $p,q \in C_e$ with $p \wedge q = 0$, one has $\mu(p + q) = \mu(p) + \mu(q)$.
    \end{enumItem}
    An \emph{$F$-valued charge} on $C_e$ is an $F$-valued signed charge $\mu$ on $C_e$ with $\mu(p) \geq 0$ for all $p \in C_e$.
    An $F$-valued signed charge $\mu$ on $C_e$ is called \emph{order bounded} if there exists $g \in F_+$ such that $\abs{\mu(p)} \leq g$ for all $p \in C_e$.
    We denote by $\ba(C_e,F)$ the set of all order bounded $F$-valued signed charges on $C_e$ and by $\ba_+(C_e,F)$ the set of all order bounded $F$-valued charges on $C_e$.
\end{definition}

In the measure-theoretic setting, the lattice structure of the space of all charges is well-known, cf. \cite[Theorem 10.53]{Aliprantis-Border-2006}.
We provide a similar result for charges on components with values in a Dedekind complete Riesz space.

\begin{theorem} \label{thm:o-bounded-signed-char-riesz-space}
    Let $E$ be a Riesz space with weak order unit $e$ and $F$ a Dedekind complete Riesz space.
    Endow $\ba(C_e,F)$ with the partial order defined by
    \[
        \mu \leq \nu \quad\text{if}\quad \nu - \mu \in \ba_+(C_e,F).
    \]
    Then:
    \begin{enumList}
        \item \label{it:thm:o-bounded-signed-char-riesz-space:riesz-space}
        $\ba(C_e,F)$ is a Riesz space, where, for all $\mu,\nu \in \ba(C_e,F)$ and $p \in C_e$, we have
        \begin{align}
            \label{eq:thm:o-bounded-signed-char-riesz-space:sup}
            (\mu \vee \nu)(p) & = \sup\Dset{\mu(q) + \nu(p - q)}{q \in C_e, q \leq p}, \\
            \label{eq:thm:o-bounded-signed-char-riesz-space:inf}
            (\mu \wedge \nu)(p) & = \inf\Dset{\mu(q) + \nu(p - q)}{q \in C_e, q \leq p}, \\
            \label{eq:thm:o-bounded-signed-char-riesz-space:mod}
            \abs{\mu}(p) & = \sup\Dset{\mu(q) - \mu(p - q)}{q \in C_e, q \leq p}, \\
            \label{eq:thm:o-bounded-signed-char-riesz-space:pos}
            \mu^+(p) & = \sup\Dset{\mu(q)}{q \in C_e, q \leq p}, \\
            \label{eq:thm:o-bounded-signed-char-riesz-space:neg}
            \mu^-(p) & = -\inf\Dset{\mu(q)}{q \in C_e, q \leq p}.
        \end{align}

        \item \label{it:thm:o-bounded-signed-char-riesz-space:dedekind-complete}
        For every net $(\mu_\alpha)$ in $\ba(C_e,F)$ and $\mu \in \ba(C_e,F)$, we have $\mu_\alpha \uparrow \mu$ if and only if $\mu_\alpha(p) \uparrow \mu(p)$ for all $p \in C_e$.
        In particular, $\ba(C_e,F)$ is Dedekind complete.
    \end{enumList}
    \end{theorem}

\begin{proof}
    \ref{it:thm:o-bounded-signed-char-riesz-space:riesz-space}
    Note that the given partial order is compatible with the vector space structure of $\ba(C_e,F)$.
    Let $\mu,\nu \in \ba(C_e,F)$.
    Since $F$ is Dedekind complete and $\mu$ and $\nu$ are order bounded, for each $p \in C_e$, the supremum
    \[
        \eta(p) \coloneqq \sup\Dset{\mu(q) + \nu(p - q)}{q \in C_e, q \leq p}
    \]
    exists in $F$.
    
    We show that $\eta \in \ba(C_e,F)$.
    Note that $\eta(0) = 0$.
    Let $p,q \in C_e$ be such that $p \wedge q = 0$.
    Let $x \in C_e$ be such that $x \leq p + q$ and consider $x_p \coloneqq x \wedge p \in C_e$ and $x_q \coloneqq x \wedge q \in C_e$.
    Then $x_p + x_q = x$, $x_p \wedge x_q = 0$, and $(p - x_p) \wedge (q - x_q) = 0$, thus
    \begin{align*}
        \mu(x) + \nu(p + q - x)
        & = \mu(x_p + x_q) + \nu(p - x_p + q - x_q) \\
        & = \mu(x_p) + \nu(p - x_p) + \mu(x_q) + \nu(q - x_q) \\
        & \leq \eta(p) + \eta(q).
    \end{align*}
    Taking the supremum over all $x \in C_e$ with $x \leq p + q$ yields $\eta(p + q) \leq \eta(p) + \eta(q)$.
    Conversely, let $y_p,y_q \in C_e$ be such that $y_p \leq p$ and $y_q \leq q$.
    Then $y_p \wedge y_q = 0$, $(p - y_q) \wedge (q - y_q) = 0$, and $y_p + y_q \leq p + q$, thus
    \begin{align*}
        \mu(y_p) + \nu(p - y_p) + \mu(y_q) + \nu(q - y_q)
        & = \mu(y_p + y_q) + \nu(p + q - (y_p + y_q)) \\
        & \leq \eta(p + q).
    \end{align*}
    Taking the supremum over all $y_p \in C_e$ with $y_p \leq p$ and then over all $y_q \in C_e$ with $y_q \leq q$ yields $\eta(p) + \eta(q) \leq \eta(p + q)$.
    We conclude that $\eta(p + q) = \eta(p) + \eta(q)$, which shows that $\eta$ is an $F$-valued signed charge on $C_e$.
    Finally, since $\mu$ and $\nu$ are order bounded, one readily checks that $\eta$ is also order bounded.
    We conclude that $\eta \in \ba(C_e,F)$.

    Next, we show that $\eta = \mu \vee \nu$.
    Note that $\mu \leq \eta$ and $\nu \leq \eta$.
    Let $\lambda \in \ba(C_e,F)$ be such that $\mu \leq \lambda$ and $\nu \leq \lambda$.
    Then, for all $p \in C_e$, we obtain
    \begin{align*}
        \eta(p)
        \leq \sup\Dset{\lambda(q) + \lambda(p - q)}{q \in C_e, q \leq p}
        = \lambda(p),
    \end{align*}
    thus $\eta \leq \lambda$.
    Hence, $\eta = \mu \vee \nu$, showing that $\ba(C_e,F)$ is a Riesz space and that \eqref{eq:thm:o-bounded-signed-char-riesz-space:sup} holds.
    Note that the equations \eqref{eq:thm:o-bounded-signed-char-riesz-space:inf}--\eqref{eq:thm:o-bounded-signed-char-riesz-space:neg} follow immediately from \eqref{eq:thm:o-bounded-signed-char-riesz-space:sup}.

    \smallskip

    \ref{it:thm:o-bounded-signed-char-riesz-space:dedekind-complete}
    Let $(\mu_\alpha)_{\alpha \in A}$ be an increasing net in $\ba_+(C_e,F)$ that is order bounded from above by some $\eta \in \ba_+(C_e,F)$.
    Since $F$ is Dedekind complete, the supremum $\mu(p) \coloneqq \sup_{\alpha \in A} \mu_\alpha(p)$ exists for all $p \in C_e$.
    We show that $\mu \in \ba(C_e,F)$.
    Note that $\mu(0) = 0$.
    Let $p,q \in C_e$ be such that $p \wedge q = 0$.
    The inequality $\mu(p + q) \leq \mu(p) + \mu(q)$ is straightforward.
    For all $\alpha,\beta \in A$, there exists $\gamma \in A$ with $\alpha \preceq \gamma$ und $\beta \preceq \gamma$ and it follows that
    \[
        \mu_\alpha(p) + \mu_\beta(q)
        \leq \mu_\gamma(p) + \mu_\gamma(q)
        = \mu_\gamma(p + q)
        \leq \mu(p + q).
    \]
    Taking the supremum over all $\alpha$ and $\beta$ yields $\mu(p) + \mu(q) \leq \mu(p + q)$.
    Hence, $\mu$ is an $F$-valued charge on $C_e$.
    Finally, since $\mu_\alpha \leq \eta$ for all $\alpha \in A$ and $\eta$ is order bounded, it follows that $\mu$ is also order bounded, thus $\mu \in \ba_+(C_e,F)$.
    One readily checks that $\mu_\alpha \uparrow \mu$.
\end{proof}

Recall that $C_e$ is a lattice, cf. \cite[Theorem 1.49]{Aliprantis-Burkinshaw-2006}.
For two lattices $X,Y$, a \emph{lattice isomorphism} is a bijective map $f \colon X \to Y$ such that, for all $x_1,x_2 \in X$, one has
\[
    f(x_1 \vee x_2) = f(x_1) \vee f(x_2) \text{ and } f(x_1 \wedge x_2) = f(x_1) \wedge f(x_2).
\]
Note that a linear map between two Riesz spaces is a Riesz isomorphism if and only if it is a linear lattice isomorphism.

\begin{theorem} \label{thm:space-of-charges-independent-of-w-o-unit}
    Let $E$ be a Riesz space with the principal projection property, $e_1,e_2 \in E$ weak order units, and $F$ a Dedekind complete Riesz space.
    \begin{enumList}
        \item \label{it:thm:space-of-charges-independent-of-w-o-unit:bijection-between-components}
        The maps
        \begin{align*}
            \theta_1 \colon C_{e_1} \to C_{e_2}, \quad p \mapsto P_p(e_2), \\
            \theta_2 \colon C_{e_2} \to C_{e_1}, \quad q \mapsto P_q(e_1),
        \end{align*}
        are mutually inverse lattice isomorphisms.

        \item \label{it:thm:space-of-charges-independent-of-w-o-unit:well-defined}
        For every $\mu \in \ba(C_{e_2},F)$, one has $\mu \circ \theta_1 \in \ba(C_{e_1},F)$ and, for every $\mu \in \ba(C_{e_1},F)$, one has $\mu \circ \theta_2 \in \ba(C_{e_2},F)$.

        \item \label{it:thm:space-of-charges-independent-of-w-o-unit:riesz-isomorphism}
        The maps
        \begin{align*}
            \Phi_1 \colon \ba(C_{e_2},F) \to \ba(C_{e_1},F), & \quad \mu \mapsto \mu \circ \theta_1, \\ 
            \Phi_2 \colon \ba(C_{e_1},F) \to \ba(C_{e_2},F), & \quad \mu \mapsto \mu \circ \theta_2,
        \end{align*}
        are mutually inverse Riesz isomorphisms.
    \end{enumList}
\end{theorem}

\begin{proof}
    \ref{it:thm:space-of-charges-independent-of-w-o-unit:bijection-between-components}
    First, we show that $\theta_1$ is a lattice homomorphism.
    To this end, let $p,q \in C_{e_1}$.
    For every $n \in \nat$, we have
    \[
        (n(p \vee q)) \wedge e_2
        = ((np) \wedge e_2) \vee ((nq) \wedge e_2)
        \leq P_p(e_2) \vee P_q(e_2)
    \]
    and taking the supremum over all $n$ yields $P_{p \vee q}(e_2) \leq P_p(e_2) \vee P_q(e_2)$.
    For all $m,n \in \nat$, we have
    \begin{align*}
        ((np) \vee e_2) \wedge ((mq) \vee e_2)
        & = ((np) \vee (mq)) \wedge e_2 \\
        & \leq ((n + m)(p \vee q)) \wedge e_2
          \leq P_{p \vee q}(e_2)
    \end{align*}
    and taking the supremum first over all $n$ and then over all $m$ yields $P_p(e_2) \vee P_q(e_2) \leq P_{p \vee q}(e_2)$.
    We conclude that
    \[
        \theta_1(p \vee q)
        = P_{p \vee q}(e_2)
        = P_p(e_2) \vee P_q(e_2)
        = \theta_1(p) \vee \theta_1(q).
    \]
    Similarly, one can verify that $\theta_1(p \wedge q) = \theta_1(p) \wedge \theta_1(q)$.
    We conclude that $\theta_1$ is a lattice homomorphism.
    Note that a similar argument shows that $\theta_2$ is a lattice homomorphism.

    Next, we show that $\theta_1$ and $\theta_2$ are mutually inverse.
    Let $p \in C_{e_1}$.
    For all $x \in B_p^d$, we have $\abs{x} \wedge P_p(e_2) = 0$, thus $x \in B_{P_p(e_2)}^d$.
    Conversely, for all $x \in B_{P_p(e_2)}^d$, we have
    \[
        0
        \leq \Abs{x} \wedge p \wedge e_2
        \leq \Abs{x} \wedge P_p(e_2)
        = 0,
    \]
    thus $\abs{x} \wedge p \wedge e_2 = 0$.
    Since $e_2$ is a weak order unit, we obtain $\abs{x} \wedge p = 0$, thus $x \in B_p^d$.
    We conclude that $B_p^d = B_{P_p(e_2)}^d$, which implies $B_p = B_{P_p(e_2)}$.
    Hence,
    \[
        \theta_2(\theta_1(p))
        = P_{P_p(e_2)}(e_1)
        = P_p(e_1)
        = p,
    \]
    which shows that $\theta_2 \circ \theta_1 = \id_{C_{e_1}}$.
    Similarly, one shows that $\theta_1 \circ \theta_2 = \id_{C_{e_2}}$.
    Hence, $\theta_1$ and $\theta_2$ are mutually inverse.

    \smallskip

    \ref{it:thm:space-of-charges-independent-of-w-o-unit:well-defined}
    Let $\mu \in \ba(C_{e_2},F)$ and note that $\mu(\theta_1(0)) = \mu(0) = 0$.
    By \ref{it:thm:space-of-charges-independent-of-w-o-unit:bijection-between-components}, we have
    \[
        \theta_1(p_1) \wedge \theta_1(p_2)
        = \theta_1(p_1 \wedge p_2)
        = \theta_1(0)
        = 0.
    \]
    From this, it follows that
    \[
        \mu(\theta_1(p_1 + p_2))
        = \mu(\theta_1(p_1 \vee p_2))
        = \mu(\theta_1(p_1) \vee \theta_1(p_2))
        = \mu(\theta_1(p_1)) + \mu(\theta_1 p_2)).
    \]
    This shows that $\mu \circ \theta_1$ is a signed $F$-valued charge on $C_{e_1}$.
    Since $\mu$ is order bounded, $\mu \circ \theta_1$ is order bounded, as well, thus $\mu \circ \theta_1 \in \ba(C_{e_1},F)$.
    Similarly, it follows that $\mu \circ \theta_2 \in \ba(C_{e_2},F)$ for all $\mu \in \ba(C_{e_1},F)$.

    \smallskip

    \ref{it:thm:space-of-charges-independent-of-w-o-unit:riesz-isomorphism}
    It is straightforward that the maps $\Phi_1$ and $\Phi_2$ are both linear and positive. 
    Since $\theta_1$ and $\theta_2$ are mutually inverse, it follows that $\Phi_1$ and $\Phi_2$ are also mutually inverse.
    Hence, $\Phi_1$ and $\Phi_2$ are mutually inverse Riesz isomorphisms.
\end{proof}

\begin{example} \label{ex:signed-charge}
    Let $(\Omega,\calA,\lambda)$ be a measure space such that $\lambda(\Omega) < \infty$.
    Similarly to \cref{ex:charge-motivation}, $\bfone \colon \omega \mapsto 1$ is a weak order unit of $L^1(\Omega,\calA,\mu)$ and the components of $\bfone$ are precisely the characteristic functions of measurable sets, i.e., $C_\bfone = \dset{\bfone_A}{A \in \calA}$.
    Let $\ba(\calA)$ be the space of all bounded real-valued signed charges on $\calA$ and note that $\ba(\calA)$ is a Dedekind complete Riesz space where, for all $\mu,\nu \in \ba(\calA)$ and $A \in \calA$, one has
    \[
        (\mu \vee \nu)(A) = \sup\Dset{\mu(B) + \nu(A \setminus B)}{B \in \calA, B \subseteq A},
    \]
    cf. \cite[Theorem 10.53]{Aliprantis-Border-2006}.
    Then the maps
    \begin{align*}
        \Phi \colon \ba(C_\bfone,\real) \to \ba(\calA), & \quad \mu \mapsto \left( A \mapsto \mu(\bfone_A) \right), \\
        \Psi \colon \ba(\calA) \to \ba(C_\bfone,\real), & \quad \nu \mapsto \left( \bfone_A \mapsto \nu(A) \right),
    \end{align*}
    are mutually inverse Riesz isomorphisms.
\end{example}

\begin{theorem}
    Let $(E,e,T)$ be a $T$-universally complete conditional Riesz triple.
    Then $\ba(C_e,R(T))$ is an $R(T)$-module with the addition and multiplication given pointwise and
    \begin{equation}
        \label{eq:thm:ba-norm:norm}
        \Norm{\cdot}_{\ba(C_e,R(T))} \colon \ba(C_e,R(T)) \to R(T), \quad \mu \mapsto \abs{\mu}(e),
    \end{equation}
    is an $R(T)$-valued norm on $\ba(C_e,R(T))$.
\end{theorem}

\begin{proof}
    It is straightforward that $\ba(C_e,R(T))$ is an $R(T)$-module and that $\Norm{\cdot}_{\ba(C_e,R(T))}$ satisfies the triangle inequality and is definite.
    We show that $\Norm{\cdot}_{\ba(C_e,R(T))}$ is absolutely homogeneous.

    Let $\mu \in \ba(C_e,R(T))$.
    For all $g \in R(T)_+$ and $p \in C_e$, we have
    \begin{equation}
        \label{eq:thm:ba-norm:pos-hom}
        \begin{aligned}
            \abs{g \mu}(p)
            & = \sup\Dset{g \mu(q) - g \mu(p - q)}{q \in C_e, q \leq p} \\
            & = g \sup\Dset{\mu(q) - \mu(p - q)}{q \in C_e, q \leq p}
              = g \abs{\mu}(p).
        \end{aligned}
    \end{equation}

    Let $f \in R(T)$.
    For all $p \in C_e$, it follows with \eqref{eq:thm:ba-norm:pos-hom} that 
    \begin{align*}
        \left( \abs{f^+ \mu} \wedge \abs{f^- \mu} \right)(p)
        & = \inf \Dset{\abs{f^+ \mu}(q) + \abs{f^- \mu}(p - q)}{p \in C_e, p \leq q} \\
        & \leq \abs{f^+ \mu}(p) \wedge \abs{f^- \mu}(p)
          = f^+ \abs{\mu}(p) \wedge f^- \abs{\mu}(p) \\
        & = (f^+ \wedge f^-) \abs{\mu}(p)
          = 0,
    \end{align*}
    thus $f^+ \mu$ and $f^- \mu$ are disjoint.
    Hence, we have 
    \begin{align*}
        \abs{f \mu}
        & = \abs{f^+ \mu - f^- \mu}
          = \abs{f^+ \mu + f^- \mu}
          = \abs{\abs{f} \mu}.
    \end{align*}
    Finally, with the aid of \eqref{eq:thm:ba-norm:pos-hom}, we calculate
    \[
        \norm{f \mu}_{\ba(C_e,R(T))}
        = \abs{f \mu}(e)
        = \abs{\abs{f} \mu}(e)
        = \abs{f} \abs{\mu}(e)
        = \abs{f} \norm{\mu}_{\ba(C_e,R(T))}.
    \]
    We conclude that $\Norm{\cdot}_{\ba(C_e,R(T))}$ is an $R(T)$-valued norm on $\ba(C_e,R(T))$.
\end{proof}

For measurable spaces $(\Omega,\calA)$, the norm of a signed charge $\mu \colon \calA \to \real$ is usually defined as the \emph{total variation}, i.e.,
\[
    \norm{\mu} = \sup \Dset{\sum_{A \in \calZ} \abs{\mu(A)}}{\calZ \subseteq \calA \text{ is a finite partition of $\Omega$}}.
\]
We show that the norm on $\ba(C_e,R(T))$, defined in \eqref{eq:thm:ba-norm:norm}, has a similar representation.
To this end, we call a finite set $\set{p_1,\dots,p_n}$ of pairwise disjoint components of $e$ a \emph{finite partition of $e$} if $\sum_{i=1}^n p_i = e$.

\begin{proposition}
    Let $(E,e,T)$ be a $T$-universally complete conditional Riesz triple and $\mu \in \ba(C_e,R(T))$.
    Then:
    \begin{equation}
        \label{eq:variation-norm}
        \norm{\mu}_{\ba(C_e,R(T))} = \sup \Dset{\sum_{p \in \calZ} \abs{\mu(p)}}{\calZ \subseteq C_e \text{ is a finite partition of $e$}}.
    \end{equation}
\end{proposition}

\begin{proof}
    First note that, for every finite partition $\calZ = \set{p_1,\dots,p_n} \subseteq C_e$ of $e$, we have
    \[
        \sum_{i=1}^n \abs{\mu(p_i)}
        \leq \sum_{i=1}^n \abs{\mu}(p_i)
        = \abs{\mu}(e).
    \]
    Since $R(T)$ is Dedekind complete, it follows that the supremum in the right-hand-side of \eqref{eq:variation-norm} exists and is bounded from above by $\abs{\mu}(e)$.
    Conversely, denoting the supremum in \eqref{eq:variation-norm} by $S$, for every $p \in C_e$, we have
    \[
        \mu(p) - \mu(e - p)
        \leq \abs{\mu(p)} + \abs{\mu(e - p)}
        \leq S. 
    \]
    Taking the supremum over all $p \in C_e$ together with the formula for $\abs{\mu}(e)$ provided in \cref{thm:o-bounded-signed-char-riesz-space} yields $\abs{\mu}(e) \leq S$.
    Hence, $\norm{\mu}_{\ba(C_e,R(T))} = \abs{\mu}(e) = S$.
\end{proof}

\begin{proposition} \label{thm:space-of-charges-independent-of-w-o-unit-isometric}
    Let $(E,e_1,T)$ be a $T$-universally complete conditional Riesz triple and let $e_2 \in E_+$ be another weak order unit with $Te_2 = e_2$.
    Then the maps
    \begin{align*}
        \Phi_1 \colon \ba(C_{e_1},R(T)) \to \ba(C_{e_2},R(T)), & \quad \mu \mapsto \left( q \mapsto \mu(P_q(e_1)) \right), \\
        \Phi_2 \colon \ba(C_{e_2},R(T)) \to \ba(C_{e_1},R(T)), & \quad \mu \mapsto \left( p \mapsto \mu(P_p(e_2)) \right),
    \end{align*}
    are well-defined mutually inverse $R(T)$-linear isometric Riesz isomorphisms.
\end{proposition}

\begin{proof}
    The $R(T)$-linearity is straightforward.
    It follows from \cref{thm:space-of-charges-independent-of-w-o-unit} that $\Phi_1$ and $\Phi_2$ are Riesz isomorphisms.
    Since $\Phi_1$ is a Riesz isomorphism, we have, for all $\mu \in \ba(C_{e_1},R(T))$, that
    \begin{align*}
        \Norm{\Phi_1(\mu)}_{\ba(C_{e_2},R(T))}
        & = \left( \Abs{\Phi_1(\mu)} \right)(e_2)
          = \left( \Phi_1(\abs{\mu}) \right) (e_2) \\
        & = \abs{\mu}(P_{e_2}(e_1))
          = \abs{\mu}(e_1)
          = \norm{\mu}_{\ba(C_{e_1},R(T))}.
    \end{align*}
    Thus, $\Phi_1$ is isometric.
    Similarly, $\Phi_2$ is isometric.
\end{proof}

By \cref{thm:space-of-charges-independent-of-w-o-unit} and \cref{thm:space-of-charges-independent-of-w-o-unit-isometric}, the structure of the Riesz space and $R(T)$-module $\ba(C_e,R(T))$ does not depend on the choice of the weak order unit $e$.

A notion for absolute continuity for positive linear operators, which is similar to the subsequent definition, was introduced in \cite{Luxemburg-Shep-1978}.

\begin{definition} \label{def:T-absolutely-continuous}
    Let $(E,e,T)$ be a conditional Riesz triple.
    We call an $R(T)$-valued signed charge $\mu \in \ba(C_e,R(T))$ \emph{$T$-absolutely continuous} if $\mu(p) \in B_{Tp}$ for all $p \in C_e$, where $B_{Tp}$ is the band generated by $Tp$ in $R(T)$.
    In this case, we write $\mu \ll T$.
    We denote
    \[
        \ba(T) \coloneqq \Dset{\mu \in \ba(C_e,R(T))}{\mu \ll T}.
    \]
\end{definition}

\begin{example}
    Let $(\Omega,\calA,\lambda)$ be a finite measure space, and consider the Riesz space $L^1(\Omega,\calA,\lambda)$ and the expectation operator
    \[
        T \colon L^1(\Omega,\calA,\lambda) \to L^1(\Omega,\calA,\lambda), \quad f \mapsto \left( \int f \,\rmd\lambda \right) \cdot \bfone,
    \]
    where $\bfone$ is the (equivalence class) of the constant function $\bfone \colon \omega \to 1$.
    Then, for all $p \in C_e$, one has
    \[
        B_{Tp} =
        \begin{cases}
            \real \cdot \bfzero = \Set{0} & \text{if } Tp = 0, \\
            \real \cdot \bfone = \Dset{\alpha \bfone}{\alpha \in \real} & \text{if } Tp > 0,
        \end{cases}
    \]
    where $\bfzero$ is the (equivalence class of the) constant function $\bfzero \colon \omega \to 0$.
    In particular, a real-valued signed charge $\mu \in \ba(\calA)$ is absolutely continuous with respect to $\lambda$ if and only if $\mu \ll T$ (with the identification of $\ba(\calA)$ and $\ba(C_e,\real)$ from \cref{ex:signed-charge}).
\end{example}

\begin{theorem}[Lebesgue decomposition]
    Let $(E,e,T)$ be a $T$-universally complete conditional Riesz triple.
    Then $\ba(T)$ is an $R(T)$-submodule of $\ba(C_e,R(T))$ and a projection band in $\ba(C_e,R(T))$.
\end{theorem}

\begin{proof}
    It is straightforward that $\ba(T)$ is an $R(T)$-submodule of $\ba(C_e,R(T))$.
    %
    Since $B_{Tp}$ is solid for all $p \in C_e$, it follows that $\ba(T)$ is a solid Riesz subspace and, thus, an order ideal of $\ba(C_e,R(T))$.

    In order to show that $\ba(T)$ is a band, we show that $\ba(T)$ is order closed.
    Let $(\mu_\alpha)$ be a net in $\ba(T)_+$ and $\mu \in \ba(C_e,R(T))$ such that $\mu_\alpha \uparrow \mu$.
    By \cref{thm:o-bounded-signed-char-riesz-space}, we then have $\mu_\alpha(p) \uparrow \mu(p)$ for all $p \in C_e$.
    Since $\mu_\alpha(p) \in B_{Tp}$ for all $p \in C_e$ and $B_{Tp}$ is order closed, it follows that $\mu(p) \in B_{Tp}$.
    Hence, $\mu \in \ba(T)$ and $\ba(T)$ is order closed and, thus, a band.

    Finally, since $\ba(C_e,R(T))$ is Dedekind complete by \cref{thm:o-bounded-signed-char-riesz-space}, $\ba(T)$ is a projection band in $\ba(C_e,R(T))$.
\end{proof}

\section{Integration with respect to \texorpdfstring{$T$}{T}-absolutely continuous \texorpdfstring{$R(T)$}{R(T)}-valued charges} \label{sec:integration}

Let $(E,e,T)$ be a $T$-universally complete conditional Riesz triple.
In this section, we develop an integration theory for elements in $L^\infty(T)$ with respect to $R(T)$-valued charges, i.e., we define the integral $\int f \,\rmd \mu$, where $f \in L^\infty(T)$ and $\mu \in \ba(T)$.
To this end, we first define \emph{$e$-step functions with $R(T)$-coefficients} and define the integral for these objects.
Afterwards, we extend the definition to $L^\infty(T)$.


\begin{definition}
    Let $(E,e,T)$ be a $T$-universally complete conditional Riesz triple.
    An element $x \in E$ is called an \emph{$e$-step function with $R(T)$-coefficients} if there exist $\alpha_1,\dots,\alpha_n \in R(T)$ and pairwise disjoint $p_1,\dots,p_n \in C_e$ such that $x = \sum_{i=1}^n \alpha_i p_i$.
    We call $\sum_{i=1}^n \alpha_i p_i$ a \emph{representation} of $x$.
    This representation is called a \emph{standard representation} if $\sum_{i=1}^n p_i = e$.
    We denote by $\calS_e(T)$ the set of all $e$-step functions with $R(T)$-coefficients.
\end{definition}

Note that, if $x = \sum_{i=1}^n \alpha_i p_i$ is a representation of $x \in \calS_e(T)$, then $x = \sum_{i=1}^n \alpha_i p_i + 0 \cdot (e - \sum_{i=1}^n p_i)$ is a standard representation of $x$.
Hence, every element of $\calS_e(T)$ has a standard representation.

\begin{proposition} \label{thm:step-functions-riesz-submodule}
    Let $(E,e,T)$ be a $T$-universally complete conditional Riesz triple, $x,y \in \calS_e(T)$ with standard representations $x = \sum_{i=1}^n \alpha_i p_i$ and $y = \sum_{j=1}^m \beta_j q_j$ and let $\gamma \in R(T)$.
    Then:
    \begin{align*}
        x + y & = \sum_{i=1}^n \sum_{j=1}^m (\alpha_i + \beta_j) (p_i \wedge q_j), \\
        \gamma x & = \sum_{i=1}^n \gamma \alpha_i p_i, \\
        \Abs{x} & = \sum_{i=1}^n \Abs{\alpha_i} p_i.
    \end{align*}
    In particular, $\calS_e(T)$ is both a Riesz subspace and $R(T)$-submodule of $L^\infty(T)$.
\end{proposition}

\begin{proof}
    The stated representations for $\gamma x$ and $\abs{x}$ are immediate.
    Moreover, we obtain
    \begin{align*}
        x + y
        & = \sum_{i=1}^n \alpha_i p_i + \sum_{j=1}^m \beta_j q_j
          = \sum_{i=1}^n \alpha_i (p_i \wedge e) + \sum_{j=1}^m \beta_j (e \wedge q_j) \\
        & = \sum_{i=1}^n \sum_{j=1}^m \alpha_i (p_i \wedge q_j) + \sum_{i=1}^n \sum_{j=1}^m \beta_j (p_i \wedge q_j)
          = \sum_{i=1}^n \sum_{j=1}^m (\alpha_i + \beta_j) (p_i \wedge q_j).
        \qedhere
    \end{align*}
\end{proof}

\begin{lemma} \label{thm:simple-integral-well-defined}
    Let $(E,e,T)$ be a $T$-universally complete conditional Riesz triple, $\mu \in \ba(T)$, and $x \in \calS_e(T)$ with two standard representations
    \[
        \sum_{i=1}^n \alpha_i p_i = x = \sum_{j=1}^m \beta_j q_j.
    \]
    Then
    \[
        \sum_{i=1}^n \alpha_i \mu(p_i) = \sum_{j=1}^m \beta_j \mu(q_j).
    \]
\end{lemma}

\begin{proof}
    Let $i \in \set{1,\dots,n}$ and $j \in \set{1,\dots,m}$.
    Note that
    \begin{align*}
        \sum_{k=1}^n \sum_{\ell=1}^m \beta_\ell (p_k \wedge q_\ell)
        & = \sum_{\ell=1}^m \beta_\ell (e \wedge q_\ell)
          = \sum_{\ell=1}^m \beta_\ell q_\ell 
          = \sum_{k=1}^n \alpha_k p_k \\
        & = \sum_{k=1}^n \alpha_k (p_k \wedge e)
          = \sum_{k=1}^n \sum_{\ell=1}^m \alpha_k (p_k \wedge q_\ell).
    \end{align*}
    Moreover, for all $k \in \set{1,\dots,n}$ and $\ell \in \set{1,\dots,m}$, one has $(p_k \wedge q_\ell) \wedge (p_i \wedge q_j) = \delta_{ik} \delta_{j\ell} (p_i \wedge q_j)$.
    It follows that
    \begin{align*}
        \alpha_i (p_i \wedge q_j)
        & = \sum_{k=1}^n \sum_{\ell=1}^m \alpha_k \left( (p_k \wedge q_\ell) \wedge (p_i \wedge q_j) \right) \\
        & = \sum_{k=1}^n \sum_{\ell=1}^m \beta_\ell \left( (p_k \wedge q_\ell) \wedge (p_i \wedge q_j) \right)
          = \beta_j (p_i \wedge q_j).
    \end{align*}
    Then $(\alpha_i - \beta_j)(p_i \wedge q_j) = 0$.
    Since $T$ is an averaging operator, it follows that $(\alpha_i - \beta_j)T(p_i \wedge q_j) = 0$, thus $(\alpha_i - \beta_j) \in B_{T(p_i \wedge q_j)}^d$.
    As $\mu(p_i \wedge q_j) \in B_{T(p_i \wedge q_j)}$, we get $(\alpha_i - \beta_j) \mu(p_i \wedge q_j) = 0$ and we conclude that $\alpha_i \mu(p_i \wedge q_j) = \beta_j \mu(p_i \wedge q_j)$.

    Finally, we obtain
    \begin{align*}
        \sum_{i=1}^n \alpha_i \mu(p_i)
        & = \sum_{i=1}^n \alpha_i \mu(p_i \wedge e)
          = \sum_{i=1}^n \alpha_i \mu\left( \sum_{j=1}^m (p_i \wedge q_j) \right) \\
        & = \sum_{i=1}^n \sum_{j=1}^m \alpha_i \mu(p_i \wedge q_j)
          = \sum_{i=1}^n \sum_{j=1}^m \beta_j \mu(p_i \wedge q_j) \\
        & = \sum_{j=1}^m \beta_j \mu\left( \sum_{i=1}^n (p_i \wedge q_j) \right)
          = \sum_{j=1}^m \beta_j \mu(e \wedge q_j)
          = \sum_{j=1}^m \beta_j \mu(q_j).
        \qedhere
    \end{align*}
\end{proof}

In the subsequent example, we demonstrate how \cref{thm:simple-integral-well-defined} fails to be true if the charge $\mu$ is not assumed to be $T$-absolutely continuous.

\begin{example}
    Consider the Riesz space $E = L^1[0,1]$ and the conditional expectation $T = \id_E$ on $E$.
    Then $R(T) = E$.
    Recall from \cref{ex:signed-charge} that $C_\bfone = \dset{\bfone_A}{A \subseteq [0,1] \text{ measurable}}$.
    Consider the function $\varphi \colon \real \to \real, t \mapsto t^2$, and
    \[
        \mu \colon C_\bfone \to E, \quad \bfone_A \mapsto \bfone_{\varphi^{-1}[A]}.
    \]
    Then $\mu \in \ba(E,R(T))$.
    Note that
    \[
        \mu\left( \bfone_{\left[0,\tfrac{1}{2}\right]} \right)
        = \bfone_{\left[0,\tfrac{1}{\sqrt{2}}\right]}
        \not\in B_{\bfone_{\left[0,\tfrac{1}{2}\right]}}
        = B_{T\bfone_{\left[0,\tfrac{1}{2}\right]}},
    \]
    which shows that $\mu$ is not $T$-absolutely continuous.

    The element $x \coloneqq \bfone_{\left[ 0,\tfrac{1}{2} \right]}$ has the two standard representations
    \[
        x = \bfone \cdot \bfone_{\left[0,\tfrac{1}{2}\right]} + \bfzero \cdot \bfone_{\left(\tfrac{1}{2},1\right]} \text{ and } x = \bfone_{\left[0,\tfrac{1}{2}\right]} \cdot \bfone_{\left[0,\tfrac{1}{2}\right]} + \bfzero \cdot \bfone_{\left(\tfrac{1}{2},1\right]},
    \]
    but
    \begin{align*}
        \bfone \cdot \mu\left( \bfone_{\left[0,\tfrac{1}{2}\right]} \right) + \bfzero \cdot \mu\left( \bfone_{\left( \tfrac{1}{2},1 \right]} \right)
        & = \bfone_{\left[0,\tfrac{1}{\sqrt{2}}\right]}, \\
        \bfone_{\left[0,\tfrac{1}{2}\right]} \cdot \mu\left( \bfone_{\left[0,\tfrac{1}{2}\right]} \right) + \bfzero \cdot \mu\left( \bfone_{\left( \tfrac{1}{2},1 \right]} \right)
        & = \bfone_{\left[0,\tfrac{1}{2}\right]} \cdot \bfone_{\left[0,\tfrac{1}{\sqrt{2}}\right]}
          = \bfone_{\left[0,\tfrac{1}{2}\right]}.
    \end{align*}
\end{example}

\begin{definition}
    Let $(E,e,T)$ be a $T$-universally complete conditional Riesz triple and let $\mu \in \ba(T)$.
    We define the map
    \[
        I_\mu \colon \calS_e(T) \to R(T), \quad \sum_{i=1}^n \alpha_i p_i \mapsto \sum_{i=1}^n \alpha_i \mu(p_i).
    \]
    For $x \in \calS_e(T)$, we call $I_\mu(x)$ the \emph{$R(T)$-valued integral of $x$ with respect to $\mu$}.
\end{definition}

\begin{lemma} \label{thm:integral-of-step}
    Let $(E,e,T)$ be a $T$-universally complete conditional Riesz triple and let $\mu \in \ba(T)$.
    Then:
    \begin{enumList}
        \item \label{it:thm:integral-of-step:extends-charge}
        $I_\mu|_{C_e} = \mu$.
        
        \item \label{it:thm:integral-of-step:linear}
        $I_\mu$ is $R(T)$-linear.
        
        \item \label{it:thm:integral-of-step:positive}
        If $\mu$ is positive, then $I_\mu$ is positive.
    \end{enumList}
\end{lemma}

\begin{proof}
    \ref{it:thm:integral-of-step:extends-charge}
    This follows directly from the definition of $I_\mu$.

    \smallskip

    \ref{it:thm:integral-of-step:linear}
    Let $x,y \in \calS_e(T)$ with standard representations $x = \sum_{i=1}^n \alpha_i p_i$ and $y = \sum_{j=1}^m \beta_j q_j$.
    Then $x + y = \sum_{i=1}^n \sum_{j=1}^m (\alpha_i + \beta_j)(p_i \wedge q_j)$ is a standard representation by \cref{thm:step-functions-riesz-submodule} and it follows that
    \begin{align*}
        I_\mu(x + y)
        & = \sum_{i=1}^n \sum_{j=1}^m (\alpha_i + \beta_j) \mu(p_i \wedge q_j) \\
        & = \sum_{i=1}^n \sum_{j=1}^m \alpha_i \mu(p_i \wedge q_j) + \sum_{i=1}^n \sum_{j=1}^m \beta_j \mu(p_i \wedge q_j) \\
        & = \sum_{i=1}^n \alpha_i \mu(p_i) + \sum_{j=1}^m \beta_j \mu(q_j)
          = I_\mu(x) + I_\mu(y).
    \end{align*}
    Let $\gamma \in R(T)$.
    Then $\gamma x = \sum_{i=1}^n \gamma \alpha_i p_i$ is a standard representation by \cref{thm:step-functions-riesz-submodule} and, as $\gamma \alpha_i \in R(T)$ for all $i \in \set{1,\dots,n}$, we obtain 
    \[
        I_\mu(\gamma x)
        = \sum_{i=1}^n \gamma \alpha_i \mu(p_i)
        = \gamma \sum_{i=1}^n \alpha_i \mu(p_i)
        = \gamma I_\mu(x).
    \]

    \smallskip

    \ref{it:thm:integral-of-step:positive}
    This follows from the fact that a step function $x \in \calS_e(T)$ with standard representation $\sum_{i=1}^n \alpha_i p_i$ is positive if and only if $\alpha_i \geq 0$ holds for all $i \in \set{1,\dots,n}$ with $p_i \neq 0$.
\end{proof}

\begin{proposition} \label{thm:integral-of-step-riesz-hom}
    Let $(E,e,T)$ be a $T$-universally complete conditional Riesz triple.
    \begin{enumList}
        \item \label{it:thm:integral-of-step-riesz-hom:well-defined}
        For every $\mu \in \ba(T)$, we have $I_\mu \in \rmL^\rmr(\calS_e(T),R(T))$.
        
        \item \label{it:thm:integral-of-step-riesz-hom:riesz}
        The map
        \[
            I \colon \ba(T) \to \rmL^\rmr(\calS_e(T),R(T)), \quad \mu \mapsto I_\mu,
        \]
        is an $R(T)$-linear Riesz homomorphism.
    \end{enumList}
\end{proposition}

\begin{proof}
    \ref{it:thm:integral-of-step-riesz-hom:well-defined}
    If $\mu \in \ba(T)$, then a calculation involving the decomposition $\mu = \mu^+ - \mu^-$ and \cref{thm:integral-of-step}\ref{it:thm:integral-of-step:extends-charge} and \ref{it:thm:integral-of-step:linear} shows that $I_\mu = I_{\mu^+} - I_{\mu^-}$.
    Hence, $I_\mu \in \rmL^\rmr(\calS_e(T),R(T))$.

    \ref{it:thm:integral-of-step-riesz-hom:riesz}
    It is straightforward that $I$ is $R(T)$-linear.
    By \cref{thm:riesz-kantorovich} and \cref{thm:o-bounded-signed-char-riesz-space}, we have
    \begin{align*}
        \abs{I_\mu}(p)
        & = \sup\Dset{I_\mu(q) + I_{\mu}(p - q)}{q \in C_e, q \leq p} \\
        & = \sup\Dset{\mu(q) + \mu(p - q)}{q \in C_e, q \leq p}
          = \abs{\mu}(p)
          = I_{\abs{\mu}}(p)
    \end{align*}
    for all $p \in C_e$, thus $I_{\abs{\mu}} = \abs{I_\mu}$.
    Hence, $I$ is a Riesz homomorphism. 
\end{proof}

\begin{definition}
    Let $(E,e,T)$ be a $T$-universally complete conditional Riesz triple.
    For $f \in L^\infty_+(T)$ and $\mu \in \ba_+(T)$, we define
    \[
        \int f \,\rmd\mu \coloneqq \sup\Dset{I_\mu(g)}{g \in \calS_e(T), 0 \leq g \leq f}.
    \]
    For $f \in L^\infty(T)$ and $\mu \in \ba(T)$, we then define
    \[
        \int f \,\rmd \mu \coloneqq \int f^+ \,\rmd\mu^+ - \int f^- \,\rmd\mu^+ - \int f^+ \,\rmd\mu^- + \int f^- \,\rmd\mu^-
    \]
    and call $\int f \,\rmd\mu$ the \emph{$R(T)$-valued integral of $f$ with respect to $\mu$}.
\end{definition}

\begin{proposition} \label{thm:conditional-sombrero}
    Let $(E,e,T)$ be a $T$-universally complete conditional Riesz triple and let $\mu \in \ba(T)_+$.
    \begin{enumList}
        \item \label{it:thm:conditional-sombrero:continuity}
        For every $f \in L^\infty(T)_+$ and every sequence $(s_n)_{n \in \bbN}$ in $\calS_e(T)_+$ that converges $u$-uniformly to $f$ for some $u \in R(T)_+$, the sequence $(I_\mu(s_n))$ converges $(u\mu(e))$-uniformly to $\int f \,\rmd\mu$.
        
        \item \label{it:thm:conditional-sombrero:approximation-exists}
        For all $f \in L^\infty(T)_+$, there exists an increasing sequence $(s_n)_{n \in \bbN}$ in $\calS_e(T)_+$ that converges $u$-uniformly to $f$ for some $u \in R(T)_+$.     
    \end{enumList}
\end{proposition}

\begin{proof}
    \ref{it:thm:conditional-sombrero:continuity}
    Since $s_n \uparrow f$ $u$-uniformly, there exists a sequence $(\eps_n)_{n \in \bbN}$ of positive real numbers such that $\eps_n \downarrow 0$ and $\abs{f - s_n} \leq \eps_n u$ for all $n \in \nat$.
    For all $n \in \bbN$, we then have $-\eps_n u + s_n \leq f \leq \eps_n u + s_n$.
    In view of \cref{thm:step-functions-riesz-submodule}, we have $(-\eps_n u + s_n)^+ \in \calS_e(T)$.
    Moreover, as $0 \leq (-\eps_n u + s_n)^+ \leq f^+ = f$, we immediately obtain
    \[
        -\eps_n u \mu(e) + I_\mu(s_n)
        = I_\mu(-\eps_n u + s_n)
        \leq I_\mu((-\eps_n u + s_n)^+)
        \leq \int f \,\rmd\mu.
    \]
    On the other hand, for every $g \in \calS_e(T)$ with $0 \leq g \leq f$, we have $g \leq \eps_n u + s_n$, thus $I_\mu(g) \leq I_\mu(\eps_n u + s_n) = \eps_n u \mu(e) + I_\mu(s_n)$ and so 
    \[
        \int f \,\rmd\mu \leq \eps_n u \mu(e) + I_\mu(s_n).
    \]
    We conclude that
    \[
        -\eps_n u \mu(e) \leq \int f \,\rmd\mu - I_\mu(s_n) \leq \eps_n u \mu(e)
    \]
    holds for all $n \in \bbN$, thus $(I_\mu(s_n))$ converges $(u\mu(e))$-uniformly to $\int f \,\rmd\mu$.
    
    \ref{it:thm:conditional-sombrero:approximation-exists}
    Let $\alpha \in R(T)_+$ be such that $f \leq \alpha$.
    Then $0 \leq \alpha^{-1} f \leq \alpha^{-1} \alpha = P_\alpha(e) \leq e$, so by \cref{thm:freudenthal}\ref{it:thm:freudenthal:uniform} there exists an increasing sequence $(t_n)_{n \in \nat}$ of positive $e$-step functions with $\real$-coefficients that converges $e$-uniformly to $\alpha^{-1} f$.
    This implies that the sequence $(s_n)_{n \in \nat}$ in $\calS_e(T)$ defined by $s_n \coloneqq \alpha t_n$ converges $(\alpha e)$-uniformly to $\alpha \alpha^{-1} f = P_\alpha(e) f$.
    Since $f \leq \alpha$, it follows from \cref{thm:prod-w/-band-proj-invariant} that $P_\alpha(e) f = f$, which implies the claim.
\end{proof}

\begin{proposition} \label{thm:integral}
    Let $(E,e,T)$ be a $T$-universally complete conditional Riesz triple and let $\mu \in \ba(T)$.
    \begin{enumList}
        \item \label{it:thm:integral:extends-simple-integral}
        For all $f \in \calS_e(T)$, we have $\int f \,\rmd\mu = I_\mu(f)$.
        
        \item \label{it:thm:integral:linear}
        $\int \cdot \,\rmd\mu \colon L^\infty(T) \to R(T)$ is $R(T)$-linear.
        
        \item \label{it:thm:integral:positive}
        If $\mu$ is positive, then $\int \cdot \,\rmd\mu$ is positive.
    \end{enumList} 
\end{proposition}

\begin{proof}
    \ref{it:thm:integral:extends-simple-integral}
    This is straightforward.

    \smallskip

    \ref{it:thm:integral:linear}
    First, suppose that $\mu$ is positive.
    We show that the integral is $R(T)_+$-homo\-geneous.
    Let $f \in L^\infty(T)_+$ and $\alpha \in R(T)_+$.
    In order to conclude that $\int \alpha f \,\rmd\mu = \alpha \int f \,\rmd\mu$, it suffices to show that
    \begin{equation} \label{eq:thm:integral:linear:pos-hom}
        \Dset{I_\mu(g)}{g \in \calS_e(T), 0 \leq g \leq \alpha f} = \alpha \Dset{I_\mu(h)}{h \in \calS_e(T), 0 \leq h \leq f}.
    \end{equation}
    Note that, for all $h \in \calS_e(T)$ with $0 \leq h \leq f$, one has $0 \leq \alpha h \leq \alpha f$ and $\alpha I_\mu(h) = I_{\mu}(\alpha h)$.
    This shows the inclusion $(\supseteq)$.
    For the converse inclusion, let $g \in \calS_e(T)$ be such that $0 \leq g \leq \alpha f$.
    By \cref{thm:prod-w/-band-proj-invariant}, we have $P_\alpha(e) g = g$. 
    Then $\alpha^{-1} g \in \calS_e(T)$,
    \[
        0
        \leq \alpha^{-1} g
        \leq \alpha^{-1} \alpha f
        = P_\alpha(e) f
        \leq f,
    \]
    and
    \[
        \alpha I_\mu(\alpha^{-1} g)
        = I_\mu(\alpha \alpha^{-1} g)
        = I_\mu(P_\alpha(e)g)
        = I_\mu(g).
    \]
    Hence, \eqref{eq:thm:integral:linear:pos-hom} is shown and we conclude that the integral is $R(T)_+$-homogeneous.

    By \cref{thm:integral-of-step}\ref{it:thm:integral-of-step:linear}, $I_\mu$ is additive.
    Then it is a straightforward application of \cref{thm:conditional-sombrero} that the integral is additive on $L^\infty(T)_+$.
    It follows that the integral is $R(T)$-linear.

    If $\mu$ is not necessarily positive, the $R(T)$-linearity follows from the decomposition $\int \cdot \,\rmd\mu = \int \cdot \,\rmd\mu^+ - \int \cdot \,\rmd\mu^-$.

    \smallskip

    \ref{it:thm:integral:positive}
    From \cref{thm:integral-of-step}\ref{it:thm:integral-of-step:positive}, the result follows.
\end{proof}

Recall from \cref{thm:integral-of-step-riesz-hom} that the integral of elements from $\calS_e(T)$ describes an $R(T)$-linear Riesz homomorphism.
Next, we show that a similar result holds for the integral of elements from $L^\infty(T)$.

\begin{proposition} \label{thm:integral-riesz-hom}
    Let $(E,e,T)$ be a $T$-universally complete conditional Riesz triple.
    \begin{enumList}
        \item \label{it:thm:integral-riesz-hom:well-defined}
        For every $\mu \in \ba(T)$, one has $\int \cdot \,\rmd \mu \in \rmL^\rmr(L^\infty(T),R(T))$.
        
        \item \label{it:thm:integral-riesz-hom:riesz}
        The map
        \[
            J \colon \ba(T) \to \rmL^\rmr(L^\infty(T),R(T)), \quad \mu \mapsto \left( J_\mu \colon f \mapsto \int f \,\rmd\mu \right),
        \]
        is an $R(T)$-linear Riesz homomorphism.
    \end{enumList}
\end{proposition}

\begin{proof}
    \ref{it:thm:integral-riesz-hom:well-defined}
    For all $\mu \in \ba(T)$, we have the decomposition $\int \cdot \,\rmd\mu = \int \cdot \,\rmd\mu^+ - \int \cdot \,\rmd\mu^-$ and the integrals with respect to the charges $\mu^+$ and $\mu^-$ are positive.
    Hence, $\int \cdot \mu \in \rmL^\rmr(L^\infty(T),R(T))$.

    \smallskip

    \ref{it:thm:integral-riesz-hom:riesz}
    From \cref{thm:integral-of-step-riesz-hom} and \cref{thm:conditional-sombrero}, it follows that $J$ is $R(T)$-linear.
    It remains to show that $J$ is a Riesz homomorphism.

    We denote by
    \[
        I \colon \ba(T) \to \rmL^\rmr(\calS_e(T),R(T)), \quad \mu \mapsto I_\mu,
    \]
    the $R(T)$-linear Riesz homomorphism from \cref{thm:integral-of-step-riesz-hom}.
    Recall from \cref{thm:integral}\ref{it:thm:integral:extends-simple-integral} that $J(\mu)|_{\calS_e(T)} = I_\mu$.

    Let $\mu \in \ba(T)$.
    Since $J$ is positive by \cref{thm:integral}\ref{it:thm:integral:positive}, we have $J_\mu^+ \leq J_{\mu^+}$.
    For the converse inequality, let $f \in L^\infty(T)_+$.
    By \cref{thm:conditional-sombrero}\ref{it:thm:conditional-sombrero:approximation-exists}, there exists an increasing sequence $(f_n)_{n \in \nat}$ in $\calS_e(T)_+$ with $f_n \uparrow f$ $u$-uniformly for some $u \in R(T)_+$.
    Fix $n \in \nat$ and let $g_n \in \calS_e(T)_+$ with $g_n \leq f_n$.
    Then we have $g_n \leq f$, thus, from \cref{thm:riesz-kantorovich},
    \[
        I_\mu(g_n)
        = J_\mu(g_n)
        \leq \sup\Dset{J_\mu(g)}{g \in L^\infty(T)_+, g \leq f}
        = J_\mu^+(f).
    \]
    Taking the supremum over all $g_n$ yields
    \[
        I_\mu^+(f_n)
        = \sup\Dset{I_\mu(g_n)}{g_n \in \calS_e(T)_+, g_n \leq f_n}
        \leq J_\mu^+(f).
    \]
    By \cref{thm:integral-of-step-riesz-hom}, $I$ is a Riesz homomorphism, thus $I_{\mu^+}(f_n) = I_\mu+(f_n) \leq J_\mu^+(f)$.
    From \cref{thm:conditional-sombrero}\ref{it:thm:conditional-sombrero:continuity}, taking the $(\mu(e)u)$-uniform limit as $n \to \infty$ shows that $J_{\mu^+}(f) \leq J_\mu^+(f)$.
    Hence, $J_{\mu^+} \leq J_\mu^+$.
    We conclude that $J$ is a Riesz homomorphism.
\end{proof}

\section{The \texorpdfstring{$T$}{T}-strong duals of \texorpdfstring{$L^1(T)$}{L1(T)} and \texorpdfstring{$L^\infty(T)$}{Linfty(T)}} \label{sec:duals}

\begin{definition}
    Let $(E,e,T)$ be a $T$-universally complete conditional Riesz triple and $p \in [1,\infty]$. 
    \begin{enumList}
        \item An $R(T)$-linear map $\varphi \colon L^p(T) \to R(T)$ is called \emph{$\norm{\cdot}_{T,p}$-bounded} if there exists $g \in R(T)_+$ such that, for all $f \in L^p(T)$, one has $\abs{\varphi(f)} \leq g \norm{f}_{T,p}$.
        
        \item The set 
        \[
            \hat{L^p(T)} \coloneqq \Dset{\varphi \colon L^p(T) \to R(T)}{\varphi \text{ is $R(T)$-linear, regular, $\norm{\cdot}_{T,p}$-bounded}}
        \]
        is called the \emph{$T$-strong dual} of $L^p(T)$.

        \item We define the map $\norm{\cdot}_{\hat{L^p(T)}} \colon \hat{L^p(T)} \to R(T)$ by
        \[
            \norm{\varphi}_{\hat{L^p(T)}} \coloneqq \inf\Dset{g \in R(T)_+}{\forall f \in L^p(T) : \abs{\varphi(f)} \leq g \norm{f}_{T,p}}.
        \]
    \end{enumList}
\end{definition}

The following lemma is known for the case $p = 2$, cf. \cite[Lemma 3.3(iii)]{Kalauch-Kuo-Watson-2024b}.

\begin{lemma} \label{thm:strong-dual-order-continuity}
    Let $(E,e,T)$ be a $T$-universally complete conditional Riesz triple, $p \in \nat$, and $\varphi \in \hat{L^p(T)}$.
    Then $\varphi$ is order continuous.
\end{lemma}

\begin{proof}
    This is an immediate consequence of the fact that $\norm{\cdot}_{T,p}$ is order continuous, see \cref{thm:p-norm-o-continuous}. 
\end{proof}

Note that $\rmL^\rmr(L^p(T),R(T))$ is an $R(T)$-module with the addition and multiplication given pointwise and a Dedekind complete Riesz space by \cref{thm:riesz-kantorovich}.

\begin{proposition} \label{thm:T-strong-dual-ideal}
    Let $(E,e,T)$ be a $T$-universally complete conditional Riesz triple and let $p \in [1,\infty]$.
    Then $\hat{L^p(T)}$ is an $R(T)$-submodule and an order ideal of $\rmL^\rmr(L^p(T),R(T))$.
\end{proposition}

\begin{proof}
    It can be seen that $\hat{L^p(T)}$ is an $R(T)$-submodule of $\rmL^\rmr(L^p(T),R(T))$.
    %
    We show that $\hat{L^p(T)}$ is solid in $\rmL^\rmr(L^p(T),R(T))$.
    Indeed, let $\varphi \in \rmL^\rmr(L^p(T),R(T))$ and $\psi \in \hat{L^p(T)}$ be such that $\abs{\varphi} \leq \abs{\psi}$.
    Choose $h \in R(T)_+$ such that $\abs{\psi(f)} \leq h \norm{f}_{T,p}$ for all $f \in L^p(T)$.
    Let $f \in L^p(T)$.
    By \cref{thm:riesz-kantorovich}, one has
    \begin{align*}
        \abs{\abs{\psi}(f)}
          \leq \abs{\psi}(\abs{f})
        & = \sup\Dset{\abs{\psi(g)}}{g \in L^p(T), \abs{g} \leq \abs{f}} \\
        & \leq \sup\Dset{h \norm{g}_{T,p}}{g \in L^p(T), \abs{g} \leq \abs{f}}
          \leq h \norm{f}_{T,p}.
    \end{align*}
    Hence,
    \[
        \abs{\varphi(f)}
        \leq \abs{\varphi}(\abs{f})
        \leq \abs{\psi}(\abs{f})
        \leq h \norm{f}_{T,p}.
    \]
    We conclude that $\varphi \in \hat{L^p(T)}$.
    Consequently, $\hat{L^p(T)}$ is an order ideal in $\rmL^\rmr(L^p(T),R(T))$.
\end{proof}

Recall that, for $f \in L^p(T)$, we denote by $\norm{f}_{T,p}^{-1}$ the canonical partial inverse of $\norm{f}_{T,p}$ in $R(T)$.

\begin{lemma} \label{thm:T-strong-dual-norm-representation}
    Let $(E,e,T)$ be a $T$-universally complete conditional Riesz triple, $p \in [1,\infty]$, and $\varphi \in \hat{L^p(T)}$.
    \begin{enumList}
        \item \label{it:thm:T-strong-dual-norm-representation:inf-attained}
        For every $f \in L^p(T)$, we have $\abs{\varphi(f)} \leq \norm{\varphi}_{\hat{L^p(T)}} \norm{f}_{T,p}$.

        \item \label{it:thm:T-strong-dual-norm-representation:representations}
        We have
        \begin{align*}
            \norm{\varphi}_{\hat{L^p(T)}}
            & = \sup\Dset{\abs{\varphi(f)}}{f \in L^p(T), \norm{f}_{T,p} \leq e} \\
            & = \sup\Dset{\abs{\varphi(f)} \cdot \norm{f}_{T,p}^{-1}}{f \in L^p(T)}.
        \end{align*}
    \end{enumList}
\end{lemma}

\begin{proof}
    \ref{it:thm:T-strong-dual-norm-representation:inf-attained}
    For every $f \in L^p(T)$, we have
    \begin{align*}
        \abs{\varphi(f)}
        & \leq \inf\Dset{g \norm{f}_{T,p}}{g \in R(T)_+, \forall h \in L^p(T) : \abs{\varphi(h)} \leq g \norm{h}_{T,p}} \\
        & \leq \left( \inf\Dset{g \in R(T)_+}{\forall h \in L^p(T) : \abs{\varphi(h)} \leq g \Norm{h}_{T,p}} \right) \norm{f}_{T,p} \\
        & = \norm{\varphi}_{\hat{L^p(T)}} \norm{f}_{T,p}.
    \end{align*}

    \ref{it:thm:T-strong-dual-norm-representation:representations}
    We denote
    \begin{align*}
        S_1 & \coloneqq \sup\Dset{\abs{\varphi(f)}}{f \in L^p(T), \norm{f}_{T,p} \leq e}, \\
        S_2 & \coloneqq \sup\Dset{\abs{\varphi(f)} \cdot \norm{f}_{T,p}^{-1}}{f \in L^p(T)}.
    \end{align*}

    For every $f \in L^p(T)$, we have that $\norm{f \norm{f}_{T,p}^{-1}}_{T,p} = \norm{f}_{T,p} \norm{f}_{T,p}^{-1} = P_{\norm{f}_{T,p}}(e) \leq e$, hence $\abs{\varphi(f)} \norm{f}_{T,p}^{-1} = \abs{\varphi(f \norm{f}_{T,p}^{-1})} \leq S_1$, thus $S_2 \leq S_1$.

    For all $f \in L^p(T)$ with $\norm{f}_{T,p} \leq e$, we have by \ref{it:thm:T-strong-dual-norm-representation:inf-attained} that $\abs{\varphi(f)} \leq \norm{\varphi}_{\hat{L^p(T)}} \norm{f}_{T,p} \leq \norm{\varphi}_{\hat{L^p}(T)}$, which yields $S_1 \leq \norm{\varphi}_{\hat{L^p(T)}}$.

    Next, we show that $\norm{\varphi}_{\hat{L^p(T)}} \leq S_2$.
    Let $f \in L^p(T)$.
    By \ref{it:thm:T-strong-dual-norm-representation:inf-attained}, we have $\abs{\varphi(f)} \leq \norm{\varphi}_{\hat{L^p(T)}} \norm{f}_{T,p}$.
    From \cref{thm:prod-w/-band-proj-invariant}, we obtain $P_{\norm{f}_{T,p}}(e) \abs{\varphi(f)} = \abs{\varphi(f)}$.
    Hence,
    \[
        \abs{\varphi(f)}
        = P_{\norm{f}_{T,p}}(e) \abs{\varphi(f)}
        = \norm{f}_{T,p} \norm{f}_{T,p}^{-1} \abs{\varphi(f)}
        \leq \norm{f}_{T,p} S_2.
    \]
    We conclude that $\norm{\varphi}_{\hat{L^p(T)}} \leq S_2$.

    Consequently, $S_2 \leq S_1 \leq \norm{\varphi}_{\hat{L^p(T)}} \leq S_2$, which gives 
    \[
        \norm{\varphi}_{\hat{L^p(T)}} = S_1 = S_2.
        \qedhere
    \]
\end{proof}

An alternative proof of the first equality in \cref{thm:T-strong-dual-norm-representation}\ref{it:thm:T-strong-dual-norm-representation:representations} can be found in \cite[Theorem 3.6.4]{L-func-ana-2025}.

\begin{theorem}
    Let $(E,e,T)$ be a $T$-universally complete conditional Riesz triple and let $p \in [1,\infty]$.
    Then $\Norm{\cdot}_{\hat{L^p(T)}}$ is an $R(T)$-valued norm on $\hat{L^p(T)}$.
\end{theorem}

\begin{proof}
    The triangle inequality of $\Norm{\cdot}_{\hat{L^p(T)}}$ follows from the definition.
    The definiteness of $\Norm{\cdot}_{\hat{L^p(T)}}$ is a consequence of \cref{thm:T-strong-dual-norm-representation}\ref{it:thm:T-strong-dual-norm-representation:inf-attained}.
    The absolute homogeneity can be verified with the representations of $\Norm{\cdot}_{\hat{L^p(T)}}$ given in \cref{thm:T-strong-dual-norm-representation}\ref{it:thm:T-strong-dual-norm-representation:representations}.
\end{proof}

Next, we give our main results.
First, we give a representation for the $T$-strong dual of $L^1(T)$.
To this end, we need the following representation of $\hat{L^2(T)}$.

\begin{theorem}[{\cite[Theorems 3.5 and 3.7]{Kalauch-Kuo-Watson-2024}}] \label{thm:dual-of-l2}
    Let $(E,e,T)$ be a $T$-universally complete conditional Riesz triple.
    \begin{enumList}
        \item For every $f \in L^2(T)$, the map
        \[
            L^2(T) \to R(T), \quad g \mapsto T(fg),
        \]
        is $R(T)$-linear and $\norm{\cdot}_{T,2}$-bounded.

        \item The map
        \[
            L^2(T) \to \hat{L^2(T)}, \quad f \mapsto \left( g \mapsto T(fg) \right),
        \]
        is an $R(T)$-linear isometric Riesz isomorphism.
    \end{enumList}
\end{theorem}

\begin{theorem} \label{thm:dual-of-l1}
    Let $(E,e,T)$ be a $T$-universally complete conditional Riesz triple.
    \begin{enumList}
        \item \label{it:thm:dual-of-l1:well-defined}
        For each $f \in L^\infty(T)$, the map
        \[
            L^1(T) \to R(T), \quad g \mapsto T(f g),
        \]
        is $R(T)$-linear and $\norm{\cdot}_{T,1}$-bounded.

        \item \label{it:thm:dual-of-l1:isomorphism}
        The map
        \[
            \Phi \colon L^\infty(T) \to \hat{L^1(T)}, \quad f \mapsto \left( g \mapsto T(f g) \right),
        \]
        is an $R(T)$-linear isometric Riesz isomorphism.
    \end{enumList}
\end{theorem}

\begin{proof}
    \ref{it:thm:dual-of-l1:well-defined}
    The additivity of the map is straightforward and $R(T)$-homogeneity follows from the averaging property of $T$.
    To see the boundedness, observe that there exists $k \in R(T)_+$ such that $\abs{f} \leq k$.
    For all $g \in L^1(T)$, we have
    \[
        \abs{T(fg)}
        \leq T(\abs{f}\abs{g})
        \leq T(k \abs{g})
        = k T(\abs{g})
        = k \norm{g}_{T,1}.
    \]

    \smallskip

    \ref{it:thm:dual-of-l1:isomorphism}
    The additivity of $\Phi$ is straightforward and the $R(T)$-homogeneity follows from the averaging property of $T$.
    Hence, $\Phi$ is $R(T)$-linear.
    Next, we show that $\Phi$ is surjective in multiple steps.
    Across all steps, let $\varphi \in \hat{L^1(T)}$ and $k \coloneqq \norm{\varphi}_{\hat{L^1(T)}}$.

    \smallskip

    \textit{Step 1: We have $\varphi|_{L^2(T)} \in \hat{L^2(T)}$ with $\norm{\varphi|_{L^2(T)}}_{\hat{L^2(T)}} \leq k$.}

    \smallskip

    By the Hölder inequality given in \cref{thm:hoelder}, we obtain for all $f \in L^2(T)$ that
    \[
        \abs{\varphi(f)}
        \leq k \norm{f}_{T,1}
        \leq k \norm{f}_{T,2} \norm{e}_{T,2}
        \leq k \norm{f}_{T,2}.
    \]
    This implies that $\varphi|_{L^2(T)}$ is $\norm{\cdot}_{T,2}$-bounded and $\norm{\varphi|_{L^2(T)}}_{\hat{L^2(T)}} \leq k$.

    \smallskip

    \textit{Step 2: There exists a unique $h \in L^2(T)$ such that, for all $f \in L^2(T)$, one has $\varphi(f) = T(fh)$. If $\varphi$ is positive, then $h$ is positive.}

    \smallskip

    This is a direct consequence of \cref{thm:dual-of-l2}.

    \smallskip

    \textit{Step 3: For every component $p$ of $e$, we have $T(p(k\abs{h} - \abs{h}^2)) \geq 0$.}

    \smallskip

    As $L^2(T)$ is an order ideal in $L^1(T)$ and $\abs{ph} \leq \abs{h}$, we have $f \coloneqq ph \in L^2(T)$.
    It follows from step 2 and the $\norm{\cdot}_{T,1}$-boundedness of $\varphi$ that
    \begin{align*}
        T(p\abs{h}^2)
        & = T(ph^2)
          = \abs{T(ph^2)}
          = \abs{T(fh)}
          = \abs{\varphi(f)} \\
        & \leq k \norm{f}_{T,1}
          = k \norm{p h}_{T,1}
          = k T(p \abs{h}),
    \end{align*}
    thus $T(p(k \abs{h} - \abs{h}^2)) \geq 0$. 

    \smallskip

    \textit{Step 4: We have $h \in L^\infty(T)$ with $\norm{h}_{T,\infty} \leq k$.}

    \smallskip

    Taking $p \coloneqq P_{(k\abs{h} - \abs{h}^2)^-}(e)$ in step 3 gives
    \[
        0
        \leq T(p(k \abs{h} - \abs{h}^2))
        = -T((k \abs{h} - \abs{h}^2)^-)
        \leq 0.
    \]
    Since $T$ is strictly positive, we obtain $(k \abs{h} - \abs{h}^2)^- = 0$, which implies $k \abs{h} - \abs{h}^2 \geq 0$, so $\abs{h}^2 \leq k \abs{h}$.
    Using the canonical partial inverse $\abs{h}^{-1} \in E^u$ of $\abs{h}$, we obtain
    \[
        \abs{h}
        = P_{\abs{h}}(e) \abs{h}
        = \abs{h}^{-1} \abs{h}^2
        \leq k \abs{h}^{-1} \abs{h}
        \leq k.
    \]
    Thus, $h \in L^\infty(T)$ and $\norm{h}_{T,\infty} \leq k$.

    \smallskip

    \textit{Step 5: We have $\Phi(h) = \varphi$.}

    \smallskip

    Recall that $\Phi(h)$ and $\varphi$ are order continuous by \cref{thm:strong-dual-order-continuity}.
    Moreover, $\Phi(h)$ and $\varphi$ agree on $L^2(T)$ by step 2.
    By \cref{thm:freudenthal} and the fact that every $e$-step function is an element of $L^2(T)$, we obtain $\Phi(h) = \varphi$.

    \smallskip

    The steps 1 to 5 show that $\Phi$ is surjective.
    For every $f \in L^\infty(T)$ with $\Phi(f) = 0$, the uniqueness in step 2 implies that $f = 0$.
    Hence, $\Phi$ is injective.
    In the proof of the surjectivity, if $\varphi$ was positive, its pre-image $h$ was also positive by step 2. 
    Thus, $\Phi$ is bipositive.
    We conclude that $\Phi$ is a Riesz isomorphism.

    It remains to show that $\Phi$ is isometric.
    Hereby, it suffices to show that the pre-image $h$ constructed in the proof of the surjectivity satisfies $\norm{h}_{T,\infty} = k$.
    We already have $\norm{h}_{T,\infty} \leq k$ by step 4.
    For every $f \in L^1(T)$, we get, using the Hölder inequality in \cref{thm:hoelder},
    \[
        \abs{\varphi(f)}
        = \abs{\Phi(h)(f)}
        = \abs{T(fh)}
        \leq T\left( \abs{fh} \right)
        = \norm{fh}_{T,1}
        \leq \norm{f}_{T,1} \norm{h}_{T,\infty}.
    \]
    Thus, $k = \norm{\varphi}_{\hat{L^1(T)}} \leq \norm{h}_{T,\infty}$.
    Hence, $\norm{h}_{T,\infty} = k$ and $\Phi$ is isometric.
\end{proof}

We apply the integration theory developed in \cref{sec:integration} to obtain a representation for the $T$-strong dual of $L^\infty(T)$.

\begin{theorem} \label{thm:dual-of-linfty}
    Let $(E,e,T)$ be a $T$-universally complete conditional Riesz triple.
    \begin{enumList}
        \item \label{it:thm:dual-of-linfty:integral-in-linfty}
        For every $\mu \in \ba(T)$, one has $\int \cdot \,\rmd\mu \in \hat{L^\infty(T)}$.

        \item \label{it:thm:dual-of-linfty:restrictions-charge}
        For every $\varphi \in \hat{L^\infty(T)}$, one has $\varphi|_{C_e} \in \ba(T)$.
        
        \item \label{it:thm:dual-of-linfty:isomorphism}
        The maps
        \begin{align*}
            \Phi \colon \ba(T) \to \hat{L^\infty(T)}, & \quad \mu \mapsto \left( f \mapsto \int f \,\rmd\mu \right), \\
            \Psi \colon \hat{L^\infty(T)} \to \ba(T), & \quad \varphi \mapsto \varphi|_{C_e},
        \end{align*}
        are mutually inverse $R(T)$-linear isometric Riesz isomorphisms.
    \end{enumList}
\end{theorem}

\begin{proof}
    \ref{it:thm:dual-of-linfty:integral-in-linfty}
    Let $\mu \in \ba(T)$.
    By \cref{thm:integral}\ref{it:thm:integral:linear}, the integral $\int \cdot \,\rmd\mu$ is $R(T)$-linear. 
    If $\mu \geq 0$, then, for all $f \in L^\infty(T)$, we have
    \[
        \Abs{\int f \,\rmd\mu}
        \leq \int \abs{f} \,\rmd\mu
        \leq \int \norm{f}_{T,\infty} \,\rmd\mu
        = \mu(e) \norm{f}_{T,\infty}.
    \]
    For arbitrary $\mu$, we then obtain
    \[
        \Abs{\int f \,\rmd\mu}
        \leq \Abs{\int f \,\rmd\mu^+} + \Abs{\int f \,\rmd\mu^-}
        \leq \mu^+(e) \norm{f}_{T,\infty} + \mu^-(e) \norm{f}_{T,\infty}
        = \abs{\mu}(e) \norm{f}_{T,\infty}
    \]
    for all $f \in L^\infty(T)$, thus $\int \cdot \,\rmd\mu \in \hat{L^\infty(T)}$.

    \smallskip

    \ref{it:thm:dual-of-linfty:restrictions-charge}
    Let $\varphi \in \hat{L^\infty(T)}$.
    Since $\varphi$ is linear, we have $\varphi(0) = 0$ and $\varphi|_{C_e}$ is additive on disjoint elements of $C_e$.
    As $\varphi$ is order bounded and $C_e \subseteq [0,e]$, it follows that $\varphi|_{C_e}$ is order bounded.
    Let $g \in R(T)$ be such that $\abs{\varphi(f)} \leq g \norm{f}_{T,\infty}$ for all $f \in L^\infty(T)$.
    For all $p \in C_e$, we then have $\varphi(p) \leq g \norm{p}_{T,\infty}$.
    By \cref{thm:inf-norm-of-component}, we get $\norm{p}_{T,\infty} \in B_{Tp}$.
    Since $B_{Tp}$ is a band in $R(T)$, it is, in particular, an algebraic ideal, thus $g \norm{p}_{T,\infty} \in B_{Tp}$, which implies $\varphi(p) \in B_{Tp}$.
    This shows that $\varphi|_{C_e}$ is $T$-absolutely continuous.
    We conclude that $\varphi|_{C_e} \in \ba(T)$.

    \smallskip

    \ref{it:thm:dual-of-linfty:isomorphism}
    It can be verified that both $\Phi$ and $\Psi$ are $R(T)$-linear and positive.
    We show that $\Phi$ and $\Psi$ are mutually inverse, which then also shows that $\Phi$ and $\Psi$ are both Riesz isomorphisms.
    Indeed, for all $\mu \in \ba(T)$ and $p \in C_e$, we have
    \[
        \Psi(\Phi(\mu))(p)
        = \Phi(\mu)(p)
        = \int p \,\rmd\mu
        = \mu(p),
    \]
    thus $\Psi \circ \Phi = \id_{\ba(T)}$.
    In order to show that $\Phi \circ \Psi = \id_{\hat{L^\infty(T)}}$, let $\varphi \in \hat{L^\infty(T)}_+$.
    For all $p \in C_e$, we get 
    \[
        \Phi(\Psi(\varphi))(p)
        = \int p \,\rmd(\Psi(\varphi))
        = \Psi(\varphi)(p)
        = \varphi(p).
    \]
    Since $\Phi(\Psi(\varphi))$ and $\varphi$ are $R(T)$-linear, it follows that $\Phi(\Psi(\varphi))$ and $\varphi$ agree on $\calS_e(T)$.
    Let $f \in L^\infty(T)_+$.
    By \cref{thm:conditional-sombrero}\ref{it:thm:conditional-sombrero:approximation-exists}, there exists an increasing sequence $(f_n)_{n \in \nat}$ in $\calS_e(T)_+$ such that $f_n \uparrow f$ $u$-uniformly for some $u \in R(T)_+$.
    Since $\varphi$ is positive, we have $\varphi(f_n) \to \varphi(f)$ $\varphi(u)$-uniformly and, thus, in order.
    \cref{thm:conditional-sombrero}\ref{it:thm:conditional-sombrero:continuity} implies $\Phi(\Psi(\varphi))(f_n) = I_{\Psi(\varphi)}(f_n) \to \int f \,\rmd(\Psi(\varphi)) = \Phi(\Psi(\varphi))(f)$ $(u(\Psi(\varphi)(e)))$-uniformly and, thus, in order.
    As $\Phi(\Psi(\varphi))(f_n) = \varphi(f_n)$, it follows that $\Phi(\Psi(\varphi))(f) = \varphi(f)$, hence $\Phi \circ \Psi = \id_{\hat{L^\infty(T)}}$.

    It remains to establish that $\Phi$ and $\Psi$ are isometric.
    Since $\Psi = \Phi^{-1}$, it is sufficient to show that $\Phi$ is isometric.
    To this end, let $\mu \in \ba(T)$.
    We have
    \begin{align*}
        \Norm{\Phi(\mu)}_{\hat{L^\infty(T)}}
        & = \sup\Dset{\Abs{\Phi(\mu)(f)}}{f \in L^\infty(T), \norm{f}_{T,\infty} \leq e} \\
        & \leq \sup\Dset{\Phi(\abs{\mu})(\abs{f})}{f \in L^\infty(T), \norm{f}_{T,\infty} \leq e} \\
        & \leq \Phi(\abs{\mu})(e)
          = \abs{\mu}(e)
          = \norm{\mu}_{\ba(T)}.
    \end{align*}
    Conversely, for all $p \in C_e$, we have $\abs{p - (e - p)} \leq e$, thus $\norm{p - (e - p)}_{T,\infty} \leq e$
    It follows that
    \begin{align*}
        \Norm{\Phi(\mu)}_{\hat{L^\infty(T)}}
        & \geq \Abs{\Phi(\mu)(p - (e - p))}
          \geq \Phi(\mu)(p - (e - p)) \\
        & = \Phi(\mu)(p) - \Phi(\mu)(e - p)
          = \mu(p) - \mu(e - p).
    \end{align*}
    By \cref{thm:o-bounded-signed-char-riesz-space}, taking the supremum over all $p$ yields $\norm{\Phi(\mu)}_{\hat{L^\infty(T)}} \geq \abs{\mu}(e) = \norm{\mu}_{\ba(T)}$.
    We conclude that $\Phi$ is isometric.
\end{proof}

\begin{remark}
    The Theorems \ref{thm:dual-of-l2}, \ref{thm:dual-of-l1}, and \ref{thm:dual-of-linfty} give a complete description of the $T$-strong duals of $L^p(T)$ for $p \in \set{1,2,\infty}$.
    We conjecture that, for $p \in (1,\infty)$ and its conjugate exponent $q$, the $T$-strong dual of $L^p(T)$ is isometrically Riesz isomorphic to $L^q(T)$ via
    \[
        L^q(T) \to \hat{L^p(T)}, \quad f \mapsto \left( g \mapsto T(fg) \right).
    \]
\end{remark}

\section{Conditional expectations constructed from countable partitions} \label{sec:application}

In this section, we fix a $\sigma$-finite measure space $(\Omega,\calA,\mu)$ and an at most countable partition $(\Omega_i)_{i \in I}$ of $\Omega$ into $\calA$-measurable sets with $\mu(\Omega_i) \in (0,\infty)$ for all $i \in I$.
Moreover, we consider the linear operator
\[
    T \colon L^1(\Omega,\calA,\mu) \to L^1(\Omega,\calA,\mu), \quad f \mapsto \sum_{i \in I} \left( \frac{1}{\mu(\Omega_i)} \int_{\Omega_i} f \,\rmd\mu \right) \bfone_{\Omega_i}.
\]
First, we collect some basic observations without proof.

\begin{proposition} \label{prop:app-cond-exp}
    \begin{enumList}
        \item $T$ is a conditional expectation operator on $L^1(\Omega,\calA,\mu)$.
        \item $L^1(T) = \dset{f \in L^0(\Omega,\calA,\mu)}{\forall i \in I : f \bfone_{\Omega_i} \in L^1(\Omega,\calA,\mu)}$.
        \item The maximal extension of $T$ to $L^1(T)$ is given by
        \[
            T \colon L^1(T) \to L^1(T), \quad f \mapsto \sum_{i \in I} \left( \frac{1}{\mu(\Omega_i)} \int_{\Omega_i} f \,\rmd\mu \right) \bfone_{\Omega_i}.
        \]
        \item $R(T) = \dset{g \in L^1(T)}{\forall i \in I : \text{$g$ is constant on $\Omega_i$ $\mu$-a.e.}}$.
        \item \label{it:prop:app-cond-exp:l-infty}
        $L^\infty(T) = \dset{f \in L^0(\Omega,\calA,\mu)}{\forall i \in I : f \bfone_{\Omega_i} \in L^\infty(\Omega,\calA,\mu)}$.
    \end{enumList}
\end{proposition}

We describe the space $L^\infty(\Omega_i,\calA_i,\mu|_{\calA_i})$ as $L^\infty(T_i)$ for an expectation operator $T_i$.
Consider the linear operator
\[
    T_i \colon L^1(\Omega_i,\calA_i,\mu|_{\calA_i}) \to L^1(\Omega_i,\calA_i,\mu|_{\calA_i}), \quad f \mapsto \left( \frac{1}{\mu(\Omega_i)} \int_{\Omega_i} f \,\rmd\mu \right) \bfone_{\Omega_i}.
\]
In the following lemma, the statements (a), (c), and (d) are straightforward.
The proof for (b) is similar to the proof of the Dedekind completeness of $L^1(\Omega_i,\calA_i,\mu|_{\calA_i})$, cf. \cite[Example 12.5(iii)]{Zaanen-1997}, and will be omitted.

\begin{lemma} \label{thm:classical-Linfty-as-abstract}
    Let $i \in I$.
    \begin{enumList}
        \item $T_i$ is an expectation operator.
        \item $L^1(\Omega_i,\calA_i,\mu|_{\calA_i})$ is $T_i$-universally complete.
        \item $R(T_i) = \real \cdot \bfone_{\Omega_i}$.
        \item $L^\infty(T_i) = L^\infty(\Omega_i,\calA_i,\mu|_{\calA_i})$.
    \end{enumList}
\end{lemma}

\begin{remark}
    Let $i \in I$.
    We denote by $L^\infty(\Omega_i,\calA_i,\mu|_{\calA_i})'$ the norm dual space of $L^\infty(\Omega_i,\calA_i,\mu|_{\calA_i})$.
    Note that $\hat{L^\infty(T_i)}$ consists of linear maps with values in $R(T_i) = \dset{\lambda \bfone_{\Omega_i}}{\lambda \in \real}$ and that $L^\infty(\Omega_i,\calA_i,\mu|_{\calA_i})'$ consists of linear maps with values in $\real$.
    Since $R(T_i)$ and $\real$ are canonically Riesz isomorphic and isometric, the $T$-strong dual $\hat{L^\infty(T_i)}$ and the norm dual $L^\infty(\Omega_i,\calA_i,\mu|_{\calA_i})'$ are also canonically Riesz isomorphic and isometric.
    More precisely, the map
    \[
        L^\infty(\Omega_i,\calA_i,\mu|_{\calA_i})' \to \hat{L^\infty(T_i)}, \quad \varphi \mapsto \left( f \mapsto \varphi(f) \bfone_{\Omega_i} \right),
    \]
    is an isometric Riesz isomorphism.
\end{remark}

Note that every $R(T_i)$-module $M$ can be viewed as an $R(T)$-module via the multiplication
\[
    g f \coloneqq (g|_{\Omega_i}) f \qquad (g \in R(T), f \in M).
\]
In particular, the $R(T_i)$-modules $L^\infty(T_i)$ and $\hat{L^\infty(T_i)}$ may be viewed as $R(T)$-modules.
Now, it is a direct consequence of \cref{prop:app-cond-exp}\ref{it:prop:app-cond-exp:l-infty} that
\begin{equation}
    \label{eq:app-decomp} 
    L^\infty(T) \cong \prod_{i \in I} L^\infty(\Omega_i,\calA_i,\mu|_{\calA_i})
\end{equation}
as $R(T)$-modules, where $\calA_i \coloneqq \dset{A \cap \Omega_i}{A \in \calA}$.
As a direct consequence of \eqref{eq:app-decomp} and \cref{thm:classical-Linfty-as-abstract}, we obtain that
\[
    L^\infty(T) \cong \prod_{i \in I} L^\infty(T_i).
\]
We will show that this decomposition transfers to the dual spaces, i.e., we show that
\[
    \hat{L^\infty(T)} \cong \prod_{i \in I} \hat{L^\infty(T_i)}.
\]
We show that the spaces are isomorphic as Riesz spaces, as $R(T)$-modules, and as $R(T)$-normed spaces.
In order to define an approriate $R(T)$-valued norm on $\prod_{i \in I} \hat{L^\infty(T_i)}$, we need the following notation.

\begin{definition} \label{def:zero-extension}
    For $i \in I$ and $f \in L^\infty(\Omega_i,\calA_i,\mu|_{\calA_i})$, we denote
    \[
        \overline{f}^i \colon \Omega \to \real, \quad \omega \mapsto
        \begin{cases}
            f(\omega) & \text{if } \omega \in \Omega_i, \\
            0 & \text{if } \omega \not\in \Omega_i.
        \end{cases}
    \]
\end{definition}

The following theorem is straightforward.

\begin{theorem}
    \begin{enumList}
        \item The map    
        \[
            \Norm{\cdot} \colon \prod_{i \in I} \hat{L^\infty(T_i)} \to R(T), \quad (\varphi_i)_{i \in I} \mapsto \sum_{i \in I} \overline{\Norm{\varphi_i}_{\hat{L^\infty(T_i)}}}^i = \sup_{i \in I} \overline{\Norm{\varphi_i}_{\hat{L^\infty(T_i)}}}^i,
        \]
        is an $R(T)$-valued norm on $\prod_{i \in I} \hat{L^\infty(T_i)}$.

        \item The map
        \[
            \Norm{\cdot} \colon \prod_{i \in I} \ba(T_i) \to R(T), \quad (\mu_i)_{i \in I} \mapsto \sum_{i \in I} \overline{\Norm{\mu_i}_{\ba(T_i)}}^i = \sup_{i \in I} \overline{\Norm{\mu_i}_{\ba(T_i)}}^i,
        \]
        is an $R(T)$-valued norm on $\prod_{i \in I} \ba(T_i)$.
    \end{enumList}
\end{theorem}

\begin{theorem} \label{thm:t-strong-dual-of-direct-sums}
    \begin{enumList}
        \item \label{it:thm:t-strong-dual-of-direct-sums:phi-well-defined}
        For all $(\varphi_i)_{i \in I} \in \prod_{i \in I} \hat{L^\infty(T_i)}$, the map
        \[
            \varphi \colon L^\infty(T) \to R(T), \quad f \mapsto \sum_{i \in I} \overline{\varphi_i(f|_{\Omega_i})}^i,
        \] 
        is an element of $\hat{L^\infty(T)}$ with $\norm{\varphi}_{\hat{L^\infty(T)}} \leq \norm{(\varphi_i)_{i \in I}}$.
        
        \item \label{it:thm:t-strong-dual-of-direct-sums:psi-well-defined}
        For $\varphi \in \hat{L^\infty(T)}$ and $i \in I$, consider the map
        \[
            \varphi_i \colon L^\infty(T_i) \to R(T_i), \quad f \mapsto \varphi\left( \overline{f}^i \middle)\right|_{\Omega_i}.
        \]
        Then $\varphi_i \in \hat{L^\infty(T_i)}$ with $\norm{(\varphi_i)_{i \in I}} \leq \norm{\varphi}_{\hat{L^\infty(T)}}$.
        
        \item \label{it:thm:t-strong-dual-of-direct-sums:inverse}
        The maps
        \begin{align*}
            \Phi \colon \prod_{i \in I} \hat{L^\infty(T_i)} \to \hat{L^\infty(T)}, & \quad (\varphi_i)_{i \in I} \mapsto \left( f \mapsto \sum_{i \in I} \overline{\varphi_i(f|_{\Omega_i})}^i \right), \\
            \Psi \colon \hat{L^\infty(T)} \to \prod_{i \in I} \hat{L^\infty(T_i)}, & \quad \varphi \mapsto \left( f \mapsto \varphi\left(\overline{f}^i \middle)\right|_{\Omega_i} \right)_{i \in I},
        \end{align*}
        are mutually inverse $R(T)$-linear isometric Riesz isomorphisms. 
    \end{enumList}
\end{theorem}

\begin{proof}
    \ref{it:thm:t-strong-dual-of-direct-sums:phi-well-defined}
    The $R(T)$-linearity and regularity of $\varphi$ is straightforward.
    For all $f \in L^\infty(T)$, we have
    \begin{align*}
        \Abs{\varphi(f)}
        & \leq \sum_{i \in I} \Abs{\overline{\varphi_i(f|_{\Omega_i})}^i}
          = \sum_{i \in I} \overline{\Abs{\varphi_i(f|_{\Omega_i})}}^i
          \leq \sum_{i \in I} \overline{\Norm{\varphi_i}_{\hat{L^\infty(T_i)}} \Norm{f|_{\Omega_i}}_{T_i,\infty}}^i \\
        & \leq \left( \sum_{i \in I} \overline{\Norm{\varphi_i}_{\hat{L^\infty(T_i)}}}^i \right) \Norm{f}_{T,\infty}
          = \Norm{(\varphi_i)_{i \in I}} \cdot \Norm{f}_{T,\infty}.
    \end{align*}
    Hence, $\varphi$ is $\norm{\cdot}_{T,\infty}$-bounded with $\norm{\varphi}_{\hat{L^\infty(T)}} \leq \norm{(\varphi_i)_{i \in I}}$. 

    \smallskip

    \ref{it:thm:t-strong-dual-of-direct-sums:psi-well-defined}
    The $R(T_i)$-linearity and regularity of $\varphi_i$ is straightforward.
    For all $f \in L^\infty(T_i)$, note that $\norm{\overline{f}^i}_{T,\infty}|_{\Omega_i} = \norm{f}_{T_i,\infty}$, thus 
    \[
        \Abs{\varphi_i(f)}
        = \Abs{\varphi\left( \overline{f}^i \right)|_{\Omega_i}}
        = \Abs{\varphi\left( \overline{f}^i \right)}|_{\Omega_i}
        \leq \left( \Norm{\varphi}_{\hat{L^\infty(T)}} \Norm{\overline{f}^i}_{T,\infty}\right)|_{\Omega_i}
        \leq \Norm{\varphi}_{\hat{L^\infty(T)}}|_{\Omega_i} \Norm{f}_{T_i,\infty}.
    \]
    It follows that $\varphi_i$ is $\norm{\cdot}_{T_i,\infty}$-bounded with $\norm{\varphi_i}_{\hat{L^\infty(T_i)}} \leq \norm{\varphi}_{\hat{L^\infty(T)}}|_{\Omega_i}$.
    In particular, we have $\overline{\norm{\varphi_i}_{\hat{L^\infty(T_i)}}}^i \leq \norm{\varphi}_{\hat{L^\infty(T)}}$, which shows that $\norm{(\varphi_i)_{i \in I}} \leq \norm{\varphi}_{\hat{L^\infty(T)}}$.

    \smallskip

    \ref{it:thm:t-strong-dual-of-direct-sums:inverse}
    It is straightforward that $\Phi$ and $\Psi$ are $R(T)$-linear and positive.
    We show that $\Phi$ and $\Psi$ are mutually inverse.
    For all $\varphi \in \hat{L^\infty(T)}$ and $f \in L^\infty(T)$, we have
    \begin{align*}
        \Phi(\Psi(\varphi))(f)
        & = \sum_{i \in I} \overline{\Psi(\varphi)_i(f|_{\Omega_i})}^i
          = \sum_{i \in I} \overline{\varphi(\overline{f|_{\Omega_i}}^i)|_{\Omega_i}}^i \\
        & = \sum_{i \in I} \varphi(f \bfone_{\Omega_i}) \bfone_{\Omega_i}
          = \sum_{i \in I} \varphi(f) \bfone_{\Omega_i}
          = \varphi(f),
    \end{align*}
    where, in the second to last step, we used the $R(T)$-homogeneity of $\varphi$.
    We conclude that $\Phi \circ \Psi = \id$.
    Moreover, for all $(\varphi_i)_{i \in I} \in \prod_{i \in I} \hat{L^\infty(T_i)}$, $j \in I$, and $f \in L^\infty(T_j)$, we have
    \[
        \Psi(\Phi((\varphi_i)_{i \in I}))_j(f)
        = \Phi((\varphi_i)_{i \in I})\left( \overline{f}^i \middle) \right|_{\Omega_j}
        = \sum_{i \in I} \left.\overline{\varphi_i\left( \overline{f}^i \middle|_{\Omega_i} \right)}^i\right|_{\Omega_j}
        = \left.\overline{\varphi_j\left( \overline{f}^i \middle|_{\Omega_j} \right)}^i\right|_{\Omega_j}
        = \varphi_j(f),
    \]
    thus $\Psi \circ \Phi = \id$.
    We conclude that $\Phi$ and $\Psi$ are mutually inverse and, thus, Riesz isomorphisms.
    Moreover, it directly follows from \ref{it:thm:t-strong-dual-of-direct-sums:phi-well-defined} and \ref{it:thm:t-strong-dual-of-direct-sums:psi-well-defined} that $\Phi$ and $\Psi$ are isometric.
\end{proof}

Recall that, by \cref{thm:dual-of-linfty}, we have the isomorphisms
\[
    \hat{L^\infty(T)} \cong \ba(T) \quad\text{and}\quad \hat{L^\infty(T_i)} \cong \ba(T_i).
\]
Together with \cref{thm:t-strong-dual-of-direct-sums}, we immediately obtain the subsequent decomposition for $\ba(T)$.

\begin{corollary} \label{thm:decomp-bat}
    \begin{enumList}
        \item \label{it:thm:decomp-bat:phi-well-defined}
        For all $(\nu_i)_{i \in I} \in \prod_{i \in I} \ba(T_i)$ the map
        \[
            \nu \colon C_{\bfone_{\Omega}} \to R(T), \quad p \mapsto \sum_{i \in I} \overline{\nu_i(p|_{\Omega_i})}^i,
        \]
        is an element of $\ba(T)$ 
        
        \item \label{it:thm:decomp-bat:psi-well-defined}
        For $\nu \in \ba(T)$ and $i \in I$, consider the map
        \[
            \nu_i \colon C_{\bfone_{\Omega_i}} \to R(T_i), \quad p \mapsto \nu\left( \overline{p}^i \middle) \right|_{\Omega_i}.
        \]
        Then $\nu_i \in \ba(T_i)$.
        
        \item \label{it:thm:decomp-bat:iso} The maps
        \begin{align*}
            \Phi \colon \prod_{i \in I} \ba(T_i) \to \ba(T), & \quad (\nu_i)_{i \in I} \mapsto \left( p \mapsto \sum_{i \in I} \overline{\nu_i(p|_{\Omega_i})}^i \right), \\
            \Psi \colon \ba(T) \to \prod_{i \in I} \ba(T_i), & \quad \nu \mapsto \left( p \mapsto \nu\left( \overline{p}^i \middle) \right|_{\Omega_i} \right)_{i \in I},
        \end{align*}
        are mutually inverse isometric $R(T)$-linear Riesz isomorphisms.
    \end{enumList}
\end{corollary}

Similarly, one can also first prove \cref{thm:decomp-bat} directly and then conclude \cref{thm:t-strong-dual-of-direct-sums} from \cref{thm:decomp-bat}.
The proof of \cref{thm:decomp-bat} is then almost identical to the proof \cref{thm:t-strong-dual-of-direct-sums} that we gave above.
We summarize the relationships between the spaces $\hat{L^\infty(T)}$, $\prod_{i \in I} \hat{L^\infty(T_i)}$, $\ba(T)$, and $\prod_{i \in I} \ba(T_i)$ in the following scheme:

\begin{center}
    \begin{tikzcd}
        \displaystyle\hat{L^\infty(T)}
        \ar[leftrightarrow]{r}{\sim}[swap]{\ref{thm:t-strong-dual-of-direct-sums}}
        \ar[leftrightarrow]{d}{\sim}[swap]{\ref{thm:dual-of-linfty}}
        & \displaystyle\prod_{i \in I} \hat{L^\infty(T_i)}
        \ar[leftrightarrow]{d}{\sim}[swap]{\ref{thm:dual-of-linfty}}
        \\ \displaystyle\ba(T)
        \ar[leftrightarrow]{r}{\sim}[swap]{\ref{thm:decomp-bat}}
        & \displaystyle\prod_{i \in I} \ba(T_i)
    \end{tikzcd}
\end{center}

\bibliographystyle{plain}
 \bibliography{bibliography}

@article{Azouzi-Ben-Amor-Cherif-Masmoudi-2024,
    AUTHOR = {Azouzi, Y. and Ben Amor, M. A. and Cherif, D.
              and Masmoudi, M.},
     TITLE = {Conditional supremum in {R}iesz spaces and applications},
   JOURNAL = {J. Math. Anal. Appl.},
  FJOURNAL = {Journal of Mathematical Analysis and Applications},
    VOLUME = {531},
      YEAR = {2024},
    NUMBER = {1},
     PAGES = {Paper No. 127738, 17},
      ISSN = {0022-247X,1096-0813},
   MRCLASS = {46A40 (46N10 91G80)},
  MRNUMBER = {4648479},
       DOI = {10.1016/j.jmaa.2023.127738},
       URL = {https://doi.org/10.1016/j.jmaa.2023.127738},
}

@article{Azouzi-Trabelsi-2017,
    AUTHOR = {Azouzi, Y. and Trabelsi, M.},
     TITLE = {{$L^p$}-spaces with respect to conditional expectation on
              {R}iesz spaces},
   JOURNAL = {J. Math. Anal. Appl.},
  FJOURNAL = {Journal of Mathematical Analysis and Applications},
    VOLUME = {447},
      YEAR = {2017},
    NUMBER = {2},
     PAGES = {798--816},
      ISSN = {0022-247X,1096-0813},
   MRCLASS = {46E30 (47B06)},
  MRNUMBER = {3573115},
MRREVIEWER = {Dorothee\ D.\ Haroske},
       DOI = {10.1016/j.jmaa.2016.10.013},
       URL = {https://doi.org/10.1016/j.jmaa.2016.10.013},
}

@article{Buskes-daPagter-vanRooij-1991,
    AUTHOR = {Buskes, G. and de Pagter, B. and van Rooij, A.},
     TITLE = {Functional calculus on {R}iesz spaces},
   JOURNAL = {Indag. Math. (N.S.)},
  FJOURNAL = {Koninklijke Nederlandse Akademie van Wetenschappen.
              Indagationes Mathematicae. New Series},
    VOLUME = {2},
      YEAR = {1991},
    NUMBER = {4},
     PAGES = {423--436},
      ISSN = {0019-3577,1872-6100},
   MRCLASS = {46A40 (06F25 46H05)},
  MRNUMBER = {1149692},
MRREVIEWER = {C.\ B.\ Huijsmans},
       DOI = {10.1016/0019-3577(91)90028-6},
       URL = {https://doi.org/10.1016/0019-3577(91)90028-6},
}

@article{Grobler-2014,
    AUTHOR = {Grobler, J.},
     TITLE = {Jensen's and martingale inequalities in {R}iesz spaces},
   JOURNAL = {Indag. Math. (N.S.)},
  FJOURNAL = {Koninklijke Nederlandse Akademie van Wetenschappen.
              Indagationes Mathematicae. New Series},
    VOLUME = {25},
      YEAR = {2014},
    NUMBER = {2},
     PAGES = {275--295},
      ISSN = {0019-3577,1872-6100},
   MRCLASS = {60G46 (60G42 60G44)},
  MRNUMBER = {3151817},
       DOI = {10.1016/j.indag.2013.02.003},
       URL = {https://doi.org/10.1016/j.indag.2013.02.003},
}

@article{Huijsmans-dePagter-1986,
    AUTHOR = {Huijsmans, C. B. and de Pagter, B.},
     TITLE = {On von {N}eumann regular {$f$}-algebras},
   JOURNAL = {Order},
    VOLUME = {2},
      YEAR = {1986},
    NUMBER = {4},
     PAGES = {403--408},
      ISSN = {0167-8094},
       DOI = {10.1007/BF00367427},
       URL = {https://doi.org/10.1007/BF00367427},
}

@article{Huijsmans-dePagter-1984,
    AUTHOR = {Huijsmans, C. B. and de Pagter, B.},
     TITLE = {Subalgebras and {R}iesz subspaces of an {$f$}-algebra},
   JOURNAL = {Proc. London Math. Soc. (3)},
    VOLUME = {48},
      YEAR = {1984},
    NUMBER = {1},
     PAGES = {161--174},
      ISSN = {0024-6115,1460-244X},
       DOI = {10.1112/plms/s3-48.1.161},
       URL = {https://doi.org/10.1112/plms/s3-48.1.161},
}

@article{Huijsmans-dePagter-1982,
    AUTHOR = {Huijsmans, C. B. and de Pagter, B.},
     TITLE = {Ideal theory in {$f$}-algebras},
   JOURNAL = {Trans. Amer. Math. Soc.},
  FJOURNAL = {Transactions of the American Mathematical Society},
    VOLUME = {269},
      YEAR = {1982},
    NUMBER = {1},
     PAGES = {225--245},
      ISSN = {0002-9947,1088-6850},
   MRCLASS = {06F25 (46A40 46J20 54C40)},
  MRNUMBER = {637036},
MRREVIEWER = {J.\ Kist},
       DOI = {10.2307/1998601},
       URL = {https://doi.org/10.2307/1998601},
}

@article{Luxemburg-Shep-1978,
    AUTHOR = {Luxemburg, W. A. J. and Schep, A. R.},
     TITLE = {A {R}adon-{N}ikod\'ym type theorem for positive operators and
              a dual},
   JOURNAL = {Nederl. Akad. Wetensch. Indag. Math.},
    VOLUME = {40},
      YEAR = {1978},
    NUMBER = {3},
     PAGES = {357--375},
      ISSN = {0019-3577},
}

@article{Kalauch-Kuo-Watson-2024,
    AUTHOR = {Kalauch, A. and Kuo, W. and Watson, B. A.},
     TITLE = {A {H}ahn-{J}ordan decomposition and {R}iesz-{F}rechet
              representation theorem in {R}iesz spaces},
   JOURNAL = {Quaest. Math.},
    VOLUME = {47},
      YEAR = {2024},
     PAGES = {S233--S246},
      ISSN = {1607-3606,1727-933X},
       DOI = {10.2989/16073606.2023.2287843},
       URL = {https://doi.org/10.2989/16073606.2023.2287843},
}

@article{Kalauch-Kuo-Watson-2024b,
    AUTHOR = {Kalauch, A. and Kuo, W. and Watson, B. A.},
     TITLE = {Strong completeness of a class of {$L^2(T)$}-type {R}iesz
              spaces},
   JOURNAL = {Proc. Amer. Math. Soc. Ser. B},
  FJOURNAL = {Proceedings of the American Mathematical Society. Series B},
    VOLUME = {11},
      YEAR = {2024},
     PAGES = {243--253},
      ISSN = {2330-1511},
   MRCLASS = {47B60 (46A40 46E40 47B65 60B12)},
  MRNUMBER = {4762686},
MRREVIEWER = {Karol\ Le\'snik},
       DOI = {10.1090/bproc/230},
       URL = {https://doi.org/10.1090/bproc/230},
}

@article{Kuo-Labuschagne-Watson-2005,
    AUTHOR = {Kuo, W. and Labuschagne, C. C. A. and Watson, B.
              A.},
     TITLE = {Conditional expectations on {R}iesz spaces},
   JOURNAL = {J. Math. Anal. Appl.},
    VOLUME = {303},
      YEAR = {2005},
    NUMBER = {2},
     PAGES = {509--521},
      ISSN = {0022-247X,1096-0813},
       DOI = {10.1016/j.jmaa.2004.08.050},
       URL = {https://doi.org/10.1016/j.jmaa.2004.08.050},
}

@article{Labuschagne-Watson-2010,
    AUTHOR = {Labuschagne, C. C. A. and Watson, B. A.},
     TITLE = {Discrete stochastic integration in {R}iesz spaces},
   JOURNAL = {Positivity},
  FJOURNAL = {Positivity. An International Mathematics Journal Devoted to
              Theory and Applications of Positivity},
    VOLUME = {14},
      YEAR = {2010},
    NUMBER = {4},
     PAGES = {859--875},
      ISSN = {1385-1292,1572-9281},
   MRCLASS = {60H05 (46G12 46N30 60B11 60G42 60G48)},
  MRNUMBER = {2741339},
MRREVIEWER = {Myung-Sin\ Song},
       DOI = {10.1007/s11117-010-0089-1},
       URL = {https://doi.org/10.1007/s11117-010-0089-1},
}

@inproceedings{Kuo-Labuschagne-Watson-2005-b,
    AUTHOR = {Kuo, W. and Labuschagne, C. C. A. and Watson, B. A.},
    TITLE = {Zero-one laws for {R}iesz space and fuzzy random variables},
    BOOKTITLE = {Proceedings of the IFSA2005},
    ADDRESS = {Beijing, China},
    PAGES = {393--397},
    YEAR = {2005},
}

@article{Kuo-Rogans-Watson-2017,
    AUTHOR = {Kuo, W. and Rogans, M. J. and Watson, B. A.},
     TITLE = {Mixing inequalities in {R}iesz spaces},
   JOURNAL = {J. Math. Anal. Appl.},
    VOLUME = {456},
      YEAR = {2017},
    NUMBER = {2},
     PAGES = {992--1004},
      ISSN = {0022-247X,1096-0813},
       DOI = {10.1016/j.jmaa.2017.07.035},
       URL = {https://doi.org/10.1016/j.jmaa.2017.07.035},
}

@article{Yosida-Hewitt-1952,
    AUTHOR = {Yosida, K. and Hewitt, E.},
     TITLE = {Finitely additive measures},
   JOURNAL = {Trans. Amer. Math. Soc.},
  FJOURNAL = {Transactions of the American Mathematical Society},
    VOLUME = {72},
      YEAR = {1952},
     PAGES = {46--66},
      ISSN = {0002-9947,1088-6850},
   MRCLASS = {27.2X},
  MRNUMBER = {45194},
MRREVIEWER = {M.\ Cotlar},
       DOI = {10.2307/1990654},
       URL = {https://doi.org/10.2307/1990654},
}

@book{Aliprantis-Border-2006,
    AUTHOR = {Aliprantis, C. D. and Border, K. C.},
     TITLE = {Infinite dimensional analysis},
   EDITION = {Third},
      NOTE = {A hitchhiker's guide},
 PUBLISHER = {Springer, Berlin},
      YEAR = {2006},
     PAGES = {xxii+703},
      ISBN = {978-3-540-32696-0; 3-540-32696-0},
}

@book{Aliprantis-Burkinshaw-2003,
    AUTHOR = {Aliprantis, C. D. and Burkinshaw, O.},
     TITLE = {Locally solid {R}iesz spaces with applications to economics},
    SERIES = {Mathematical Surveys and Monographs},
    VOLUME = {105},
   EDITION = {Second},
 PUBLISHER = {American Mathematical Society, Providence, RI},
      YEAR = {2003},
     PAGES = {xii+344},
      ISBN = {0-8218-3408-8},
       DOI = {10.1090/surv/105},
       URL = {https://doi.org/10.1090/surv/105},
}

@book{Aliprantis-Burkinshaw-2006,
    AUTHOR = {Aliprantis, C. D. and Burkinshaw, O.},
     TITLE = {Positive operators},
      NOTE = {Reprint of the 1985 original},
 PUBLISHER = {Springer, Dordrecht},
      YEAR = {2006},
     PAGES = {xx+376},
      ISBN = {978-1-4020-5007-7; 1-4020-5007-0},
       DOI = {10.1007/978-1-4020-5008-4},
       URL = {https://doi.org/10.1007/978-1-4020-5008-4},
}

@book{Luxemburg-Zaanen-1971,
    AUTHOR = {Luxemburg, W. A. J. and Zaanen, A. C.},
     TITLE = {Riesz spaces. {V}ol. {I}},
    SERIES = {North-Holland Mathematical Library},
 PUBLISHER = {North-Holland Publishing Co., Amsterdam-London; American
              Elsevier Publishing Co., Inc., New York},
      YEAR = {1971},
     PAGES = {xi+514},
   MRCLASS = {46A40 (06A65)},
  MRNUMBER = {511676},
MRREVIEWER = {S.\ J.\ Bernau},
}

@book{Toland-2020,
    AUTHOR = {Toland, J.},
    TITLE = {The dual of {$L_{\infty}(X,\mathcal{L},\lambda)$}, finitely
              additive measures and weak convergence---a primer},
    SERIES = {SpringerBriefs in Mathematics},
 PUBLISHER = {Springer, Cham},
      YEAR = {2020},
     PAGES = {x+99},
      ISBN = {978-3-030-34731-4; 978-3-030-34732-1},
   MRCLASS = {46-02 (28A25 46B04 46E30 46T99)},
  MRNUMBER = {4292258},
MRREVIEWER = {Rigoberto\ Vera Mendoza},
       DOI = {10.1007/978-3-030-34732-1},
       URL = {https://doi.org/10.1007/978-3-030-34732-1},
}

@book{Zaanen-1997,
    AUTHOR = {Zaanen, A. C.},
     TITLE = {Introduction to operator theory in {R}iesz spaces},
 PUBLISHER = {Springer-Verlag, Berlin},
      YEAR = {1997},
     PAGES = {xii+312},
      ISBN = {3-540-61989-5},
   MRCLASS = {47B60 (46A40 47B65)},
  MRNUMBER = {1631533},
MRREVIEWER = {Yu.\ A.\ Abramovich and C.\ D.\ Aliprantis},
       DOI = {10.1007/978-3-642-60637-3},
       URL = {https://doi.org/10.1007/978-3-642-60637-3},
}

@misc{L-func-ana-2025,
    author = {Kikianty, E. and Messerschmiedt, M. and Naude, L. and Roelands, M. and Schwanke, C. and van Amstel, W. and van der Walt, J. H. and Wortel, M.},
    title = {$\mathbb{L}$-functional analysis},
    year = {2025},
    note = {https://arxiv.org/abs/2403.10222v2},
}
  
\end{document}